\documentclass[preprint, final, 12pt]{elsarticle}
\newcommand{\eval}[2][\right]{\relax
  \ifx#1\right\relax \left.\fi#2#1\rvert}

\usepackage{graphicx}
\usepackage{amssymb}
\usepackage{amsmath}
\usepackage{graphicx}
\usepackage{subfig}
\usepackage{algorithmic}
\usepackage{algorithm}
\usepackage{color}
\usepackage{booktabs, makecell, multirow, tabularx}
\usepackage[framed]{ntheorem}
\usepackage{wasysym}
\usepackage{lineno,hyperref}
\usepackage{setspace}
\usepackage{txfonts}
\theoremstyle{plain}

\theoremstyle{plain}
\newtheorem{theorem}{Theorem}
\newtheorem{lemma}{Lemma}

\theoremstyle{remark} 
\newtheorem{proof}{Proof}
\journal{XXXX XXXXXX}

\begin{document}
\begin{frontmatter}
\title{Three-term Recurrence Relations with Arbitrary Degree Steps for Orthogonal Polynomials}
\author{Bo Yang\corref{cor1}} 
\ead{bo.yang@hotmail.fr}
\cortext[cor1]{Corresponding author}

%% Author affiliation
\affiliation{organization={School of Automation, Northwestern Polytechnical University}, %Department and Organization
	addressline={127 Youyi road}, 
	city={Xi'an},
	postcode={710072}, 
	state={Shaanxi},
	country={P.R. China}}
	
\begin{abstract}
An approach to generate three-term recurrence relations with arbitrary degree steps is proposed for orthogonal polynomials. Specifically, given any class of orthogonal polynomials $\{Q_{p}(x)\}_{p=0}^{\infty}$ defined by Favard's theorem, we employ the adjacent members $Q_{p}(x)$ and $Q_{p-1}(x)$ to compute $Q_{p+s}(x)$ of high degree and the one of low degree $Q_{p-t}(x)$, where $(s,t)$ are parameters for degree step adjustment.The coefficients of both relations are analyzed, revealing novel properties that enable the derivation of three-term recurrence relations with respect to $Q_{p+s}(x)$, $Q_{p}(x)$ and $Q_{p-t}(x)$ by eliminating $Q_{p-1}(x)$. Furthermore, in addition to the standard recursive formula, which is characterized by degree increase, the formulas for degree decrease and end-to-middle directions are also formulated. Moreover, explicit recurrence relations with 2-degree steps are presented for Hermite, Gegenbauer and Legendre polynomials. The computation precision of the proposed recurrence relations is also compared with that of the standard ones.
\end{abstract}
% MSC2020 classification: 33C45, 33D45, 42C05
\end{frontmatter}
% \maketitle
% \tableofcontents

\section{Introduction}
\subsection{Recurrence relation for orthogonal polynomials}
% Orthogonal polynomials are widely applied to numerical analysis and scientific computing. The literature indicates that the research on orthogonal polynomials is primarily focused on these aspects, such as definition domain \cite{marcellan2001}, \cite{castillo2014}, zeros distribution \cite{beardon2011}, inner product type \cite{marcellan2014},  \cite{garza2021}, measure/weight \cite{cvetkovic2018}, \cite{liu2021}, \cite{gautschi2022}, \cite{milovanovic2023}, multiple orthogonal polynomials\cite{filipuk2015}, \cite{swiderski2022}, exceptional orthogonal polynomials \cite{duran2015} and semi-classical orthogonal polynomials \cite{dzhamay2022}.
Recurrence relations play central and multifaceted roles in the theory of orthogonal polynomials. They not only facilitate efficient computation as they directly show, but also support the definitions and determine the intrinsic properties of the polynomials. So far, most orthogonal polynomials, especially those which are hierarchically classified by the Askey scheme, satisfy the following three-term recurrence relation \cite{chihara1978}
\begin{equation}
	\label{3term}
	\begin{aligned}
	xQ_{p}(x)=a_{p}Q_{p+1}(x)&+b_{p}Q_{p}(x)+c_{p}Q_{p-1}(x),\; p=0,1,2,\cdots	                                                      
	\end{aligned}
\end{equation}
with the initial conditions $Q_{-1}(x)=0$ and $Q_{0}(x)=\mathit{const.}\neq 0$, where $\{a_{p}\}_{p\geq 0}$, $\{b_{p}\}_{p\geq 0}$ and $\{c_{p}\}_{p\geq 0}$ are sequences of complex numbers such that $a_{p}c_{p+1}\neq 0$ for all $p=0,1,2,\cdots$.

% unit circile orthogonal polynomials
Recurrence relations are closely intertwined with diverse research directions. As in the case of orthogonal polynomials on the real line, those defined on the unit circle possess a rich and systematic framework in which recurrence relations, though structurally different from their real‑line counterparts, serve as powerful tools for characterizing core properties \cite{simon2005p1, simon2005p2}. Several studies have sought to bridge the two settings by revealing connections or even equivalences between them. For instance, three‑term recurrence relations have been derived for orthonormal Laurent polynomials on the unit circle, and these relations, together with the associated Christoffel-Darboux formula, lead to a Favard‑type theorem \cite{barroso2005}. Such a theorem directly establishes the existence of a unique nontrivial probability measure on the unit circle from the given three‑term recurrence relation \cite{castillo2014}. Furthermore, it has been shown that a complete correspondence exists between the theory of orthogonal polynomials on the unit circle (both real and complex cases) and the polynomial systems defined by general three‑term recurrences whose coefficients involve chain sequences \cite{costa2013}. Moreover, the Riemann-Hilbert analysis has been effectively employed to construct orthogonal polynomials on the unit circle with H\"{o}lder‑type weights, enabling a transparent derivation of the Verblunsky coefficients and their essential properties \cite{branquinho2026}.

% Sobolev-type orthogonal polynomials
Orthogonal polynomials with respect to non‑classical measures typically give rise to higher‑order recurrence relations \cite{hermoso2023}. A prominent example is Sobolev‑type orthogonal polynomial, whose inner product involves derivatives \cite{marcellan2015}. Early attempts to extend Favard’s theorem to this setting have shown that such polynomials satisfy high‑order recurrence relations \cite{marcellan2001}. In the case of discrete Sobolev‑type inner products, the recurrence coefficients can be obtained via connection coefficients relating the sequences orthogonal to the Sobolev‑type inner product to those orthogonal to the underlying measure \cite{marcellan2014}. For instance, Huertas {\it et} al. derived connection formulas between monic orthogonal polynomials with respect to the Freud-Sobolev inner product and the classical Freud polynomials, and consequently obtained a five‑term recurrence relation \cite{huertas2019}. Additionally, symmetry properties of the weight function can facilitate the derivation of recurrence formulas; e.g., Garza {\it et} al. established both a five‑term and a three‑term recurrence relations with rational coefficients for a class of Sobolev‑type orthogonal polynomials \cite{garza2021}.

% multiple orthgonal polynomials
Multiple orthogonal polynomials are polynomials in one variable that satisfy orthogonality conditions with respect to more than one positive measure  \cite{assche2001}. The exploration of recurrence relations for such orthogonal polynomials becomes a significant aspect. Filipuk {\it et} al. developed a method to obtain the nearest‑neighbor recurrence coefficients from the step‑line recurrence coefficients, and also proposed the way to compute the step‑line recurrence coefficients from the coefficients determined by each measure of multiple orthogonality  \cite{filipuk2015}. Barroso {\it et} al. derived recurrence relations for multiple orthogonal polynomials on the unit circle merely based on orthogonality conditions, resulting in an appropriate definition of multiple Verblunsky coefficients and a multiple version of the Szeg\"{o} recurrence relation \cite{barroso2015}. The recent work on this theme showed that the Szeg\"{o} recurrence relation is in analogy with the nearest neighbour recurrence relations from the real line counterpart \cite{vaktnas2024}. A four‑term recurrence relation led to the analysis of asymptotic properties for the special cases, i.e., multiple orthogonal Hermite polynomials \cite{aptekarev2022}.
 
Recurrence relations serve as both effective analytical tools and fundamental structural properties in the study of  $d$-orthogonal polynomials, a class closely related to multiple orthogonal polynomials. Douak and Maroni obtained a new family of 2‑orthogonal polynomials as a particular solution to the nonlinear system governing the recurrence coefficients of classical 2‑orthogonal polynomials \cite{douak2021}. In a subsequent paper, they addressed the integral representation problem for the associated linear functionals \cite{douak2022}. What's more, for $d$-orthogonal polynomials associated with special functions, their recurrence relations can often be explicitly determined. For example, Chaggara and Gahami showed that polynomials defined by a Brenke‑type generating function and satisfying a standard four‑term recurrence relation are definitely sequences of 2‑orthogonal polynomials \cite{chaggara2025}.

% exceptional orthogonal polynomials
Introduced in the last two decades, exceptional orthogonal polynomials have been an active area for orthogonal polynomials \cite{ullate2009, ullate2012}.
Recurrence relation gives the obvious facilitation in studying such orthogonal polynomials and the more general frame, multi-indexed orthogonal polynomials. Dur\'{a}n derived higher‑order recurrence relations for exceptional Charlier, Hermite, Meixner, and Laguerre polynomials. The author also showed that the order of the obtained recurrence relation is minimal \cite{duran2015}. Odake proved that the solvable quantum mechanical systems, whose eigenfunctions are described by the multi-indexed orthogonal polynomials of Laguerre, Jacobi, Wilson and Askey-Wilson types, satisfy the generalized closure relations \cite{odake2016}.  Recurrence relation can bridge the exceptional and multi-indexed orthogonal polynomials in particular conditions, e.g., Horv\'{a}th defined multi‑index polynomials of the second kind and further explored some properties of such polynomials according to the recurrence relations of exceptional orthogonal polynomials \cite{horvath2019}.

Besides the research directions mentioned above, recurrence relations are fundamental tools for the study of orthogonal matrix polynomials \cite{duran2005}, as reflected in aspects such as the zeros of polynomials \cite{duran1999}, the specific properties \cite{rivero2020}, the asymptotic behavior of recurrence coefficients \cite{duran2001}, and the newly introduced “finite” orthogonal matrix polynomials \cite{lekesiz2025}. Moreover, recurrence relations are also employed in the analysis of coefficient perturbations. For further details, one can refer to \cite{ismail2019, shukla2023a, shukla2023b, kim2024}.
% perturbation recurrence relation
%The study of coefficient perturbations in recurrence relations is also an interesting direction. $R_{I}$ and $R_{II}$ are typical recurrence relations within this branch, and co‑recursive, co‑dilated, and co‑modified polynomials also belong to this scope \cite{kim2024}. Ismail and Ranga discussed a special $R_{II}$‑type recurrence relation that always leads to a positive measure on the unit circle \cite{ismail2019}. Shukla and Swaminathan analyzed perturbations of the $R_{II}$‑type recurrence relation using a transfer matrix method, and the effect of perturbations on the unit circle was also discussed \cite{shukla2023a}. The same authors studied the $R_{I}$-type recurrence relation when the coefficients are modified, investigating the structural relation between the original and perturbed polynomials from spectral properties \cite{shukla2023b}.

In this paper, we focus on the classical three‑term recurrence relation from Favard's theorem and extend it to a more generalized form. To be more precise, in order to overcome the limitation of computing the polynomials one by one, we propose novel three-term recurrence relations to directly compute the target polynomial from any two members of the sequence. In addition, the novel recurrence relations are applicable to all classes of orthogonal polynomials defined by Favard's theorem. The main result is summarized in Subsection \ref{sub_mainres} below.

\subsection{Main result}
\label{sub_mainres}
For a sequence of orthogonal polynomials $\{Q_{p}(x)\}_{p=0}^{\infty}$ defined by Favard's theorem, we propose an approach to develop a three-term recurrence relation with arbitrary degree steps, which is formulated by the equation below
\begin{equation}
	\label{mainres}
	Q_{p+s}(x)=M(x)Q_{p}(x)+N(x)Q_{p-t}(x),\;p\geq t\; \mathrm{and }\;(s,t)\in \mathbb{N^{+}}
\end{equation}
where $(s, t)$ are degree steps respectively higher and lower than current degree $p$. $M(x)$ denotes the coefficient for $Q_{p}(x)$; while $N(x)$ represents the one for $Q_{p-t}(x)$. Both are functions with respect to the variable $x$ and the intrinsic parameters. 

Eq. (\ref{mainres}) demonstrates that, given any two members in $\{Q_{p}(x)\}_{p=0}^{\infty}$, we can directly compute any other member of the sequence. Furthermore, it also implies that any polynomial in the sequence can be expressed in terms of two members, \textit{almost} independent of the others.

\subsection{Organization}
Section \ref{sec_preknow} presents some preliminary knowledge with respect to Favard's theorem and classical orthogonal polynomials. Section \ref{sec_arbrecurr} gives the detailed development and analysis of the three-term recurrence relation of arbitrary degree steps. As the extensions of recurrence relation proposed in Section \ref{sec_arbrecurr}, the ones for different computational directions are proposed in Section \ref{sec_direction}. In addition, the explicit three-term recurrence relations with 2-degree steps for classical orthogonal polynomials, such as Hermite, Gegenbauer and Legendre polynomials, are formulated in Section \ref{sec_inst2}. Finally, Section \ref{sec_exp} provides the experiments to demonstrate the computation precision of the proposed recurrence relations in comparison with the standard ones.

\section{Preliminary knowledge}
\label{sec_preknow}
\subsection{Favard's theorem}
\label{sec2}
It is necessary to revisit Favard's theorem first, because it serves as the preliminary knowledge of this paper. Favard's theorem gives 
the general definitions for orthogonal polynomials by means of recurrence relation. It is formulated below \cite{abramowitz1965}, \cite{schoutens2000}.
\begin{theorem}[Favard's Theorem]
\label{thm1}
Let $A_{p}$, $B_{p}$ and $C_{p}$ be arbitrary sequences of real numbers, and let \{$Q_{p}(x)$\} be defined by the recurrence relation:
\begin{equation}
	\label{Favard}
	Q_{p+1}(x)=\left(A_{p}x+B_{p}\right)Q_{p}(x)-C_{p}Q_{p-1}(x), \quad p \geq 0
\end{equation}
together with $Q_{0}(x)=1$ and $Q_{-1}(x)=0$. Then \{$Q_{p}(x)$\} is a system of orthogonal polynomials if and only if $A_{p}\neq 0$, $C_{p}\neq 0$, and $C_{p}A_{p}A_{p-1}>0$ for all $p$.
\end{theorem}
For convenience, we refer to $(A_{p}, B_{p}, C_{p})$ as the Favard's coefficients in the remainder of this paper. 

\subsection{Orthogonal polynomials}
According to Favard's Theorem, a variety of orthogonal polynomials can be generated from specific Favard's coefficients and initial conditions. Moreover, the orthogonal polynomial $Q_{p}(x)$ satisfies the orthogonality as follows:
\begin{equation}
	\label{orth}
	\int_{\Omega}Q_{p}(x)Q_{q}(x)w(x)dx=h_{p}\delta_{pq},
\end{equation}
where $\Omega$ denotes the orthogonality interval, $w(x)$ represents the weight function with $w(x)\geq 0$. $h_{p}$ is the normalization coefficient determined by
\begin{equation}
	\label{hp}
	h_{p}=\int_{\Omega}w(x)Q_{p}^{2}(x)dx,
\end{equation}
and Kronecker delta $\delta_{pq}$ has the explicit expression as
\begin{equation}
	\label{kron}
	\delta_{pq}=\left\{
	\begin{aligned}
	 1,\quad \mathrm{if}\; p=q\\
	 0,\quad \mathrm{if}\; p\neq q\\
	\end{aligned}\right..
\end{equation}

We take Hermite, Gegenbauer and Legendre polynomials as examples for further analysis and discussion in this paper. The Hermite polynomial of degree $p$ is defined by
\begin{equation}
	\label{hermite}
	H_{p}(x)=(-1)^{p}e^{x^{2}}\frac{d^{p}}{dx^{p}}e^{-x^{2}}.
\end{equation}
Hermite polynomials are orthogonal over the interval $(-\infty,\infty)$. The orthogonality is formulated as
\begin{equation}
	\label{herorth}
	\int_{-\infty}^{\infty}H_{p}(x)H_{q}(x)e^{-x^{2}}dx=2^{p}p!\sqrt{\pi}\delta_{pq}.
\end{equation}
The three-term recursive formula for Hermite polynomials is 
\begin{equation}
	\label{herrecur}
	H_{p+1}(x)=2xH_{p}(x)-2pH_{p-1}(x),\quad p\geq 1
\end{equation}
with $H_{0}(x)=1$ and $H_{1}(x)=2x$.

Since Gegenbauer polynomials have a parameter $\lambda$, the polynomials have two different forms according to the value of $\lambda$. One is corresponding to $\lambda\neq 0$, which is the most common form; the other is corresponding to $\lambda=0$. We only discuss the former case, i.e., $\lambda\neq 0$ in this paper. The explicit expression for the $p^{\mathrm{th}}$ degree Gegenbauer polynomial is
\begin{equation}
	\label{gegen}
	G_{p}^{(\lambda)}(x)=\frac{1}{\Gamma(\lambda)}\sum_{q=0}^{[p/2]}(-1)^{q}\frac{\Gamma(p-q+\lambda)}{q!(p-2q)!}(2x)^{p-2q},\;\lambda>-\frac{1}{2},
\end{equation}
where $\Gamma(\cdot)$ denotes the Gamma function. Gegenbauer polynomials are orthogonal over the interval $(-1,1)$:
\begin{equation}
	\label{gegorth}
	\int_{-1}^{1}G_{p}^{(\lambda)}(x)G_{q}^{(\lambda)}(x)(1-x^{2})^{\lambda-\frac{1}{2}}dx=\frac{\pi\Gamma(p+2\lambda)}{2^{2\lambda-1}(p+\lambda)p!\Gamma^{2}(\lambda)}\delta_{pq}.
\end{equation}
The three-term recurrence relation for Gegenbauer polynomials is 
\begin{equation}
	\label{gegrecur}
	G_{p+1}^{(\lambda)}(x)=\frac{2(p+\lambda)}{p+1}xG_{p}^{(\lambda)}(x)-\frac{p+2\lambda-1}{p+1}G_{p-1}^{(\lambda)}(x),\quad p\geq 1,
\end{equation} 
with $G_{0}^{(\lambda)}(x)=1$ and $G_{1}^{(\lambda)}(x)=2\lambda x$. 

The Legendre polynomial of degree $p$ is defined as
\begin{equation}
	\label{legendre}
	L_{p}(x)=\frac{1}{2^{p}p!}\frac{d^{p}}{dx^{p}}(x^{2}-1)^{p}.
\end{equation}
Legendre polynomials are orthogonal over the interval $(-1,1)$, and their orthogonality is formulated by
\begin{equation}
	\label{legorth}
	\int_{-1}^{1}L_{p}(x)L_{q}(x)dx=\frac{2}{2p+1}\delta_{pq}.
\end{equation}
The three-term recurrence relation for Legendre polynomials is formulated as follows:
\begin{equation}
	\label{legrecur}
	L_{p+1}(x)=\frac{2p+1}{p+1}xL_{p}(x)-\frac{p}{p+1}L_{p-1}(x),\quad p\geq 1,
\end{equation}
with $L_{0}(x)=1$ and $L_{1}(x)=x$.  

\section{Recurrence relation with arbitrary degree steps}
\label{sec_arbrecurr}
\subsection{Degree increase case}
The case for degree increase means that we use two adjacent polynomials to derive the one of higher degree. For the sequence of orthogonal polynomials $\{Q_{p}(x)\}$, (\ref{Favard}) in Favard's Theorem \ref{thm1} can be rewritten as
\begin{equation}
\label{s1}
Q_{p+1}(x)=S_{p}(x)Q_{p}(x)-C_{p}Q_{p-1}(x),\; p\geq 1,
\end{equation}
with
\begin{equation}
\label{Sdef}
S_{p}(x)=A_{p}x+B_{p}, \quad p\geq 1.
\end{equation}
The degree $p$ starts from $1$ instead of $0$, because $Q_{-1}(x)=0$ which makes $Q_{0}(x)=1$ usually serve as the first member of the polynomials.

Given any degree $p$, we first derive $Q_{p+s}(x)$ from two adjacent members, $Q_{p}(x)$ and $Q_{p-1}(x)$. $s\geq 1$ denotes the degree step in this degree increase stage. Accordingly, we represent the coefficient of $Q_{p}(x)$ by $\Psi_{p+s}^{(p)}(x)$, and the one of  $Q_{p-1}(x)$ by $\Psi_{p+s}^{(p-1)}(x)$. Here, the superscript index, e.g., $p$ in $\Psi_{p+s}^{(p)}(x)$ means that this coefficient is for $Q_{p}(x)$, while the subscript index $p+s$ implies the degree of the polynomial $Q_{p+s}(x)$ to be computed. Subsequently, we analyze the rules according to which $\Psi_{p+s}^{(p)}(x)$ and $\Psi_{p+s}^{(p-1)}(x)$ vary with respect to the degree step $s$. 

When $s=1$, this case corresponds to (\ref{s1}), thus
\begin{equation}
	\label{s1coef}
	s=1:
	\left\{
	\begin{aligned}
	\Psi_{p+1}^{(p)}(x)&=S_{p}(x)\\
	\Psi_{p+1}^{(p-1)}(x)&=-C_{p}\\
	\end{aligned}
	\right..
\end{equation}

When $s=2$, substituting $p+2$  for $p+1$ in (\ref{s1}) yields
\begin{equation}
	\label{s2recurr}
	Q_{p+2}(x)=S_{p+1}(x)Q_{p+1}(x)-C_{p+1}Q_{p}(x).
\end{equation}
After substituting (\ref{s1}) into (\ref{s2recurr}), we get
\begin{equation}
	\label{s2}
    Q_{p+2}(x)=\big(S_{p+1}(x)S_{p}(x)-C_{p+1}\big)Q_{p}(x)-C_{p}S_{p+1}(x)Q_{p-1}(x).
\end{equation}
So
\begin{equation}
	\label{s2coef}
	s=2:
	\left\{
	\begin{aligned}
	\Psi_{p+2}^{(p)}(x)&=S_{p+1}(x)S_{p}(x)-C_{p+1}\\
	\Psi_{p+2}^{(p-1)}(x)&=-C_{p}S_{p+1}(x)\\
    \end{aligned}
     \right..
\end{equation}

When $s=3$, $Q_{p+3}(x)$ can be expressed by dint of (\ref{s1}).
\begin{equation}
	\label{s3recurr}
	Q_{p+3}(x)=S_{p+2}(x)Q_{p+2}(x)-C_{p+2}Q_{p+1}(x).
\end{equation}
Substituting (\ref{s2}) and (\ref{s1}) into (\ref{s3recurr}), we have the following expression:
\begin{equation}
	\label{s3}
	\begin{aligned}
	Q_{p+3}(x)=&\big(S_{p+2}(x)S_{p+1}(x)S_{p}(x)-C_{p+1}S_{p+2}(x)-C_{p+2}S_{p}(x)\big)Q_{p}(x)\\
	                         -&\big(C_{p}S_{p+2}(x)S_{p+1}(x)-C_{p+2}C_{p}\big)Q_{p-1}(x).
	\end{aligned}
\end{equation}
Obviously, the coefficients are
\begin{equation}
	\label{s3coef}
    	s=3:
	    \left\{
	     \begin{aligned}
          \Psi_{p+3}^{(p)}(x)&=S_{p+2}(x)S_{p+1}(x)S_{p}(x)-C_{p+1}S_{p+2}(x)-C_{p+2}S_{p}(x)\\
	      \Psi_{p+3}^{(p-1)}(x)&=-\big(C_{p}S_{p+2}(x)S_{p+1}(x)-C_{p+2}C_{p}\big)\\
	     \end{aligned}
	     \right..
\end{equation}

According to the above analysis, it is not difficult to summarize the recursive relations for the coefficients $\Psi_{p+s}^{(p)}(x)$ and $\Psi_{p+s}^{(p-1)}(x)$.
We formulate them in matrix form as
\begin{equation}
\label{uprecurr}
\left[
\begin{array}{l}
	\Psi_{p+s}^{(p)}(x)\\
	\Psi_{p+s}^{(p-1)}(x)\\
	\end{array}
	\right]=
	\left[
	\begin{array}{cc}
		\Psi_{p+s-1}^{(p)}(x)  &  	\Psi_{p+s-2}^{(p)}(x) \\
	    \Psi_{p+s-1}^{(p-1)}(x)  &  	\Psi_{p+s-2}^{(p-1)}(x) \\
	    \end{array}
	    \right]
	    \left[
	    \begin{array}{c}
	    	S_{p+s-1}(x)\\
	    	-C_{p+s-1}\\
	    		    \end{array}
	    	\right],\; s\geq 3
\end{equation}
with the initial conditions denoted by (\ref{s1coef}) and (\ref{s2coef}). 

Once $\Psi_{p+s}^{(p)}(x)$ and $\Psi_{p+s}^{(p-1)}(x)$ are obtained,  $Q_{p+s}(x)$ can be computed directly from $Q_{p}(x)$ and $Q_{p-1}(x)$. We summarize $Q_{p+s}(x)$ computation, or degree increase computation in the following theorem.
\begin{theorem}
If the orthogonal polynomials $\{Q_{p}(x)\}$ follow the recurrence relation denoted by (\ref{s1}), then any polynomial
$Q_{p+s}(x)$ of higher degree can be computed from the polynomials $Q_{p}(x)$ and $Q_{p-1}(x)$ by
\begin{equation}
\label{upexpre}
Q_{p+s}(x)=\Psi_{p+s}^{(p)}(x)Q_{p}(x)+\Psi_{p+s}^{(p-1)}(x)Q_{p-1}(x),\; p\geq 1\;\mathrm{and}\; s\geq 1
\end{equation}
where the coefficients $\Psi_{p+s}^{(p)}(x)$ and $\Psi_{p+s}^{(p-1)}(x)$ satisfy the recursive relations formulated by (\ref{uprecurr}).
\end{theorem}

It should be noted that similar relations were used to prove zero interlacing property of orthogonal polynomials \cite{beardon2011}\cite{jooste2021}. However, our objective in constructing this relation is to compute the coefficients $\Psi_{p+s}^{(p)}(x)$ and $\Psi_{p+s}^{(p-1)}(x)$, and to explore their properties as well.

The coefficients $\Psi_{p+s}^{(p)}(x)$ and $\Psi_{p+s}^{(p-1)}(x)$ corresponding to $Q_{p+s}(x)$ have some properties summarized by the following lemma.
\begin{lemma}
\label{wspro}
If the sequence of orthogonal polynomials $\{Q_{p}(x)\}$ satisfies the relation formulated by (\ref{upexpre}), then there exist three
conclusions:
	
	(a) No $x$ satisfies both $\Psi_{p+s}^{(p)}(x)=0$ and $\Psi_{p+s}^{(p-1)}(x)=0$.
	
	(b) No $x$ satisfies both $\Psi_{p+s-1}^{(p)}(x)=0$ and $\Psi_{p+s}^{(p)}(x)=0$.
	
	(c) No $x$ satisfies both $\Psi_{p+s-1}^{(p-1)}(x)=0$ and $\Psi_{p+s}^{(p-1)}(x)=0$.
\end{lemma}
\begin{proof}
\label{profws}
We rewrite the recursive computation of $Q_{p+s}(x)$ and $Q_{p+s-1}(x)$ in form of matrix, i.e.
\begin{equation}
	\label{upexprmat}
	\left[
	\begin{array}{c}
		Q_{p+s}(x)\\
		Q_{p+s-1}(x)\\
	\end{array}
	\right]=\left[
	\begin{array}{cc}
		S_{p+s-1}(x)  &  -C_{p+s-1}\\
		1                         &    0\\
	\end{array}
	\right]\cdots
	\left[
	\begin{array}{cc}
		S_{p}(x)  &  -C_{p}\\
		1                         &    0\\
	\end{array}
	\right]
	\left[
	\begin{array}{c}
		Q_{p}(x)\\
		Q_{p-1}(x)\\
	\end{array}
	\right].
\end{equation}
The coefficient matrix $\mathbf{W}$ is therefore equal to
\begin{equation}
	\label{upw}
	\mathbf{W}=\left[
	\begin{array}{cc}
		S_{p+s-1}(x)  &  -C_{p+s-1}\\
		1                         &    0\\
	\end{array}
	\right]\cdots
	\left[
	\begin{array}{cc}
		S_{p+1}(x)  &  -C_{p+1}\\
		1                         &    0\\
	\end{array}
	\right]
	\left[
	\begin{array}{cc}
		S_{p}(x)  &  -C_{p}\\
		1                         &    0\\
	\end{array}
	\right].
\end{equation}
$C_{p}\neq 0$  for $p\geq 0$ is always satisfied according to Theorem \ref{thm1}. So,
\begin{equation}
	\label{detelem}
	\mathrm{det}\left(\left[
	\begin{array}{cc}
		S_{p}(x) & -C_{p}\\
		1        &    0\\
		\end{array}
	\right]\right)=C_{p},
\end{equation}
which determines
\begin{equation}
	\label{detw}
	\mathrm{det}(\mathbf{W})=\prod_{i=0}^{s-1}C_{p+i}\neq 0.
\end{equation}
On the other hand, $\mathbf{W}$ is substantially equal to
\begin{equation}
	\label{wexplict}
	\mathbf{W}=\left[
	\begin{array}{cc}
		\Psi_{p+s}^{(p)}(x)  &  \Psi_{p+s}^{(p-1)}(x)\\
		\Psi_{p+s-1}^{(p)}(x)  &   \Psi_{p+s-1}^{(p-1)}(x)\\
		\end{array}
	\right].
\end{equation}

When $\Psi_{p+s}^{(p)}(x)=0$ and $\Psi_{p+s}^{(p-1)}(x)=0$, this yields $\mathrm{det}(\mathbf{W})=0$. Likewise, both $\Psi_{p+s-1}^{(p)}(x)=0$ and $\Psi_{p+s-1}^{(p-1)}(x)=0$ still lead to $\mathrm{det}(\mathbf{W})=0$. Both results contradict (\ref{detw}), and these contradictions prove (a) in Lemma \ref{wspro}. 

When $\Psi_{p+s}^{(p)}(x)=0$ and $\Psi_{p+s-1}^{(p)}(x)=0$,  the first column of $\mathbf{W}$ is completely $0$, leading to $\mathrm{det}(\mathbf{W})=0$. This contradicts (\ref{detw}), thereby proving (b) in Lemma \ref{wspro}. 

Likewise, when $\Psi_{p+s}^{(p-1)}(x)=0$ and $\Psi_{p+s-1}^{(p-1)}(x)=0$,  the second column of $\mathbf{W}$ is completely $0$, resulting in $\mathrm{det}(\mathbf{W})=0$. This contradicts (\ref{detw}), which fact proves (c) in Lemma \ref{wspro}.  

Therefore, the proof of Lemma \ref{wspro} is completed.
\end{proof}

\subsection{Degree decrease case}
The case of degree decrease refers to computing a lower-degree polynomial from two adjacent polynomials. For the sequence of orthogonal polynomials $\{Q_{p}(x)\}$, we need to derive $Q_{p-t}(x)$ from $Q_{p}(x)$ and $Q_{p-1}(x)$. Here, $t\geq 1$ denotes the degree step in the degree decrease stage.
 
 We rewrite (\ref{s1}) into
 \begin{equation}
 	Q_{p-1}(x)=\frac{S_{p}(x)}{C_{p}}Q_{p}(x)-\frac{1}{C_{p}}Q_{p+1}(x),
 \end{equation}
and define the variables
 \begin{equation}
 	\label{tp}
 	T_{p}(x)=\frac{S_{p}(x)}{C_{p}},
 \end{equation}
 and
 \begin{equation}
 	\label{dp}
 	D_{p}=\frac{1}{C_{p}}.
 \end{equation}
 We then obtain the general recurrence relation for the degree decrease case
 \begin{equation}
 	\label{trecur}
 	Q_{p-1}(x)=T_{p}(x)Q_{p}(x)-D_{p}Q_{p+1}(x),
 \end{equation}
 which has a similar form to that for the degree increase case.
 
Likewise, the coefficient of $Q_{p}(x)$ is denoted as $\psi_{p-t}^{(p)}(x)$, while the one of $Q_{p-1}(x)$ is represented as  $\psi_{p-t}^{(p-1)}(x)$. It is necessary to analyze the rules according to which the coefficients $\psi_{p-t}^{(p)}(x)$ and $\psi_{p-t}^{(p-1)}(x)$ change with respect to $t$. We perform an analysis similar to the one for the degree increase case to summarize the variations of these coefficients.

When $t=1$, it is a special case beyond the expression of (\ref{trecur}). The following relation is established without a doubt.
\begin{equation}
	\label{t1}
	Q_{p-1}(x)=1\cdot Q_{p-1}(x)+0\cdot Q_{p}(x).
\end{equation}
The coefficients can be readily recognized as
\begin{equation}
	\label{t1coef}	
     t=1:\left\{
     \begin{aligned}
     \psi_{p-1}^{(p)}(x)&=0\\
     \psi_{p-1}^{(p-1)}(x)&=1\\
    \end{aligned}
    \right..
\end{equation}

When $t=2$, substituting $p-2$ for $p-1$ in (\ref{trecur}) yields 
\begin{equation}
	\label{t2recurr}
	Q_{p-2}(x)=T_{p-1}(x)Q_{p-1}(x)-D_{p-1}Q_{p}(x).
\end{equation}
It is obvious that the coefficients are
\begin{equation}
	\label{t2coef}
	t=2:
	\left\{
	\begin{aligned}
	\psi_{p-2}^{(p)}(x)&=-D_{p-1}\\
	\psi_{p-2}^{(p-1)}(x)&=T_{p-1}(x)\\
    \end{aligned}
     \right..
\end{equation}

When $t=3$, $Q_{p-3}(x)$ is expressed recursively according to (\ref{trecur}) as
\begin{equation}
	\label{t3recurr}
    Q_{p-3}(x)=T_{p-2}(x)Q_{p-2}(x)-D_{p-2}Q_{p-1}(x).
\end{equation}
After substituting (\ref{t2recurr}) into (\ref{t3recurr}), we get
\begin{equation}
	\label{t3}
	Q_{p-3}(x)=\big(T_{p-2}(x)T_{p-1}(x)-D_{p-2}\big)Q_{p-1}(x)-D_{p-1}T_{p-2}(x)Q_{p}(x).
\end{equation}
So, the coefficients are
\begin{equation}
	\label{t3coef}
	t=3:
	\left\{
	\begin{aligned}
	\psi_{p-3}^{(p)}(x)&=-D_{p-1}T_{p-2}(x)\\
	\psi_{p-3}^{(p-1)}(x)&=T_{p-2}(x)T_{p-1}(x)-D_{p-2}\\
    \end{aligned}
     \right..
\end{equation}

We further analyze the case where $t=4$ to clarify the rules. According to (\ref{trecur}), one can derive $Q_{p-4}(x)$ as
 \begin{equation}
 	\label{t4recurr}
 	Q_{p-4}(x)=T_{p-3}(x)Q_{p-3}(x)-D_{p-3}Q_{p-2}(x).
 \end{equation}
With the help of (\ref{t3}) and (\ref{t2recurr}), this polynomial is finally simplified as
\begin{equation}
	\label{t4}
	\begin{aligned}
     Q_{p-4}(x)=&\big(T_{p-3}(x)T_{p-2}(x)T_{p-1}(x)-D_{p-2}T_{p-3}(x)-D_{p-3}T_{p-1}(x)\big)Q_{p-1}(x)\\
                           -&\big(D_{p-1}T_{p-3}(x)T_{p-2}(x)-D_{p-3}D_{p-1}\big)Q_{p}(x)\\
	\end{aligned}.
\end{equation}
Similarly, the coefficients are
\begin{equation}
	\label{t4coef}
	t=4:
	\left\{
	\begin{aligned}
		\psi_{p-4}^{(p)}(x)&=-\big(D_{p-1}T_{p-3}(x)T_{p-2}(x)-D_{p-3}D_{p-1}\big)\\
		\psi_{p-4}^{(p-1)}(x)&=T_{p-3}(x)T_{p-2}(x)T_{p-1}(x)-D_{p-2}T_{p-3}(x)-D_{p-3}T_{p-1}(x)\\
	\end{aligned}
	\right..
\end{equation}

The above analysis helps us summarize the rule for the update of both coefficients.
\begin{equation}
	\label{downrecurr}
	\left[
	\begin{array}{l}
		\psi_{p-t}^{(p-1)}(x)\\
		\psi_{p-t}^{(p)}(x)\\
	\end{array}
	\right]=
	\left[
	\begin{array}{cc}
		\psi_{p-t+1}^{(p-1)}(x)  &  	\psi_{p-t+2}^{(p-1)}(x) \\
		\psi_{p-t+1}^{(p)}(x)  &  	\psi_{p-t+2}^{(p)}(x) \\
	\end{array}
	\right]
	\left[
	\begin{array}{c}
		T_{p-t+1}(x)\\
		-D_{p-t+1}\\
	\end{array}
	\right],\; t\geq 3
\end{equation}
with the initial conditions given in (\ref{t1coef}) and (\ref{t2coef}).

Once the coefficients are available, the computation of $Q_{p-t}(x)$ becomes straightforward. We summarize the process in the following theorem.
\begin{theorem}\label{gencase}
If the orthogonal polynomials $\{Q_{p}(x)\}$ follow the recurrence relation given in (\ref{s1}), then any polynomial $Q_{p-t}(x)$ of lower degree
can be computed from the polynomials $Q_{p}(x)$ and $Q_{p-1}(x)$ by
\begin{equation}
	\label{downexpre}
	Q_{p-t}(x)=\psi_{p-t}^{(p)}(x)Q_{p}(x)+\psi_{p-t}^{(p-1)}(x)Q_{p-1}(x),\;  t \geq 1\;\mathrm{and} \; p\geq t,
\end{equation}
where the coefficients $\psi_{p-t}^{(p)}(x)$ and $\psi_{p-t}^{(p-1)}(x)$ satisfy the recursive relations formulated by (\ref{downrecurr}).
\end{theorem}

The coefficients $\psi_{p-t}^{(p)}(x)$ and $\psi_{p-t}^{(p-1)}(x)$ for polynomial $Q_{p-t}(x)$ have similar properties described in the following lemma.
\begin{lemma}
\label{wtpro}
If the sequence of orthogonal polynomials $\{Q_{p}(x)\}$ supports the relation formulated by (\ref{downexpre}), then there exist three
conclusions:

(a) No $x$ satisfies both $\psi_{p-t}^{(p)}(x)=0$ and $\psi_{p-t}^{(p-1)}(x)=0$.

(b) No $x$ satisfies both $\psi_{p-t+1}^{(p)}(x)=0$ and $\psi_{p-t}^{(p)}(x)=0$.

(c) No $x$ satisfies both $\psi_{p-t+1}^{(p-1)}(x)=0$ and $\psi_{p-t}^{(p-1)}(x)=0$.
\end{lemma}

The proof of Lemma \ref{wtpro} is based on the idea of proof of Lemma\ref{wspro}.
\begin{proof}
The computation of $Q_{p-t}(x)$ and $Q_{p-t+1}(x)$ can be rewritten in matrix form according to (\ref{t2recurr}) 
	\label{profwt}
	\begin{equation}
		\label{downexprmat}
		\left[
		\begin{array}{c}
			Q_{p-t}(x)\\
			Q_{p-t+1}(x)\\
		\end{array}
		\right]=\left[
		\begin{array}{cc}
			T_{p-t+1}(x)  &  -D_{p-t+1}\\
			1                         &    0\\
		\end{array}
		\right]\cdots
		\left[
		\begin{array}{cc}
			T_{p-1}(x)  &  -D_{p-1}\\
			1                         &    0\\
		\end{array}
		\right]
		\left[
		\begin{array}{c}
			Q_{p-1}(x)\\
			Q_{p}(x)\\
		\end{array}
		\right].
	\end{equation}
	The coefficient matrix $\mathbf{V}$ is equal to
	\begin{equation}
		\label{downv}
		\mathbf{V}=\left[
		\begin{array}{cc}
			T_{p-t+1}(x)  &  -D_{p-t+1}\\
			1                         &    0\\
		\end{array}
		\right]\cdots
		\left[
		\begin{array}{cc}
			T_{p-2}(x)  &  -D_{p-2}\\
			1                         &    0\\
		\end{array}
		\right]
		\left[
		\begin{array}{cc}
			T_{p-1}(x)  &  -D_{p-1}\\
			1                         &    0\\
		\end{array}
		\right].
	\end{equation}
	According to Theorem \ref{thm1} and $(\ref{dp})$, we are sure that $D_{p}\neq 0$  for $p\geq 0$. So, the determinant of $\mathbf{V}$ is
	\begin{equation}
		\label{detv}
		\mathrm{det}(\mathbf{V})=\prod_{i=0}^{t-2}D_{p-i-1}\neq 0.
	\end{equation}
	Meanwhile, $\mathbf{V}$ is substantially equal to
	\begin{equation}
		\label{vexplict}
		\mathbf{V}=\left[
		\begin{array}{cc}
			\psi_{p-t}^{(p-1)}(x)  &  \psi_{p-t}^{(p)}(x)\\
			\psi_{p-t+1}^{(p-1)}(x)  &   \psi_{p-t+1}^{(p)}(x)\\
		\end{array}
		\right].
	\end{equation}
	
	Both $\psi_{p-t}^{(p)}(x)=0$ and $\psi_{p-t}^{(p-1)}(x)=0$ yield $\mathrm{det}(\mathbf{V})=0$, and so do $\psi_{p-t+1}^{(p)}(x)=0$ and $\psi_{p-t+1}^{(p-1)}(x)=0$. The results contradict (\ref{detv}), and these contradictions prove (a) in Lemma \ref{wtpro}. 
	
	When $\psi_{p-t}^{(p)}(x)=0$ and $\psi_{p-t+1}^{(p)}(x)=0$,  the second column of $\mathbf{V}$ is completely $0$, leading to $\mathrm{det}(\mathbf{V})=0$. This contradicts (\ref{detv}), thereby proving (b) in Lemma \ref{wtpro}. 
	
	Likewise, when $\psi_{p-t}^{(p-1)}(x)=0$ and $\psi_{p-t+1}^{(p-1)}(x)=0$,  the first column of $\mathbf{V}$ is completely $0$, resulting in $\mathrm{det}(\mathbf{V})=0$. This contradicts (\ref{detv}), which fact proves (c) in Lemma \ref{wtpro}.  
	
	So, the proof of Lemma \ref{wtpro} is completed.
\end{proof}
\subsection{Generalized recurrence relation}
According to (\ref{upexpre}) and (\ref{downexpre}), we can  construct an equation set in form of matrix like
\begin{equation}
	\label{genform}
    \left[
    \begin{array}{l}
	Q_{p+s}(x)\\
	Q_{p-t}(x)\\
    \end{array}
    \right]=
    \mathbf{K}
     \left[
     \begin{array}{c}
	  Q_{p}(x)\\
	  Q_{p-1}(x)\\
     \end{array}
      \right],
\end{equation} 
in which the coefficient matrix is defined as
\begin{equation}
	\label{coeffmat}
	\mathbf{K}=\left[
	\begin{array}{cc}
		\Psi_{p+s}^{(p)}(x)  &  	\Psi_{p+s}^{(p-1)}(x) \\
		\psi_{p-t}^{(p)}(x)     &   	\psi_{p-t}^{(p-1)}(x) \\
	\end{array}
	\right].
\end{equation}

Obviously, we can establish a relation between (\ref{upexpre}) and (\ref{downexpre}), taking $Q_{p-1}(x)$ as the bridge
\begin{equation}
\label{fracrelat}
\frac{Q_{p+s}(x)-\Psi_{p+s}^{(p)}(x)Q_{p}(x)}{\Psi_{p+s}^{(p-1)}(x)}=\frac{Q_{p-t}(x)-\psi_{p-t}^{(p)}(x)Q_{p}(x)}{\psi_{p-t}^{(p-1)}(x)},
\end{equation}
where $\Psi_{p+s}^{(p-1)}(x)\neq 0$ and $\psi_{p-t}^{(p-1)}(x)\neq 0$.  Further simplification of (\ref{fracrelat}) will generate a three-term recurrence relation with respect to $Q_{p+s}(x)$, $Q_{p}(x)$ and $Q_{p-t}(x)$. We summarize it in the following theorem, which serves as the main contribution of this paper.
\begin{theorem}
\label{mainthm}
If the orthogonal polynomials $Q_{p}(x)$ follow the recurrence relation denoted by (\ref{Favard}), then they satisfy a three-term recurrence relation with arbitrary degree steps:
\begin{equation}
\label{abi_rec}
Q_{p+s}(x)=M(x)Q_{p}(x)+N(x)Q_{p-t}(x),\;  1 \leq t\leq p, \; \mathrm{and}\; s\geq 1
\end{equation}
with
\begin{equation}
\label{stepcoeff}
\left\{
\begin{aligned}
	M(x)= &\frac{\mathrm{det}(\mathbf{K})}{\psi_{p-t}^{(p-1)}(x)}\\
	N(x)= &\frac{\Psi_{p+s}^{(p-1)}(x)}{\psi_{p-t}^{(p-1)}(x)}\\
	\end{aligned}
	\right..
\end{equation}
\end{theorem}

Note that both $M(x)$ and $N(x)$ are substantially functions with respect to the variable $x$, the intrinsic parameters $\mathbf{\Theta}$ of the polynomials, and the coefficients $(A_{i},B_{i},C_{i})$ defined in (\ref{Favard}) with $p+s-1\geq i\geq p-t+1$. In general, given any class of orthogonal polynomials, the coefficients $(A_{i},B_{i},C_{i})$ can be easily determined according to Favard's theorem. Consequently, we are able to acquire the new recurrence relation with arbitrary degree steps set by $(s,t)$ for this kind of orthogonal polynomials. 

 \subsection{Analysis of special cases}
 \label{sub_spec}
 (\ref{abi_rec}) is established under the condition that both $\Psi_{p+s}^{(p-1)}(x) \neq 0$ and $\psi_{p-t}^{(p-1)}(x) \neq 0$ in (\ref{fracrelat}). When these conditions are not satisfied, it will fall into two special cases, for which we discuss the corresponding resolutions in this subsection. 
  
 The first case is that $\Psi_{p+s}^{(p-1)}(x)=0$ . In this condition, (\ref{upexpre}) reduces to
 \begin{equation}
    \label{upexpres1}
    Q_{p+s}(x)=\Psi_{p+s}^{(p)}(x)Q_{p}(x).
  \end{equation}
 It is straightforward to extend (\ref{upexpres1}) to a three-term recurrence relation
\begin{equation}
   \label{abi_recs2}
   Q_{p+s}(x)=\Psi_{p+s}^{(p)}(x)Q_{p}(x)+0\cdot Q_{p-t}(x),
\end{equation}
which explicitly assigns a coefficient with permanent $0$ value for $Q_{p-t}(x)$. This means substantially that $Q_{p+s}(x)$ is independent of $Q_{p-t}(x)$ in this case. Moreover, no matter whether $\psi_{p-t}^{(p-1)}(x)$ is equal to $0$ or not, (\ref{abi_recs2}) is always valid as long as $\Psi_{p+s}^{(p-1)}(x)=0$. 
 
The second special case is when $\Psi_{p+s}^{(p-1)}(x)\neq 0$ but $\psi_{p-t}^{(p-1)}(x)=0$.  (\ref{upexpre}) remains valid in this case; however (\ref{downexpre}) reduces to 
 \begin{equation}
    \label{downexpres1}
	 Q_{p-t}(x)=\psi_{p-t}^{(p)}(x)Q_{p}(x).
\end{equation}
$Q_{p-1}(x)$ is the key for computing $Q_{p+s}(x)$ as formulated in (\ref{upexpre}).  Nevertheless, it is impossible to rely on $Q_{p-t}(x)$ to compute $Q_{p-1}(x)$ due to (\ref{downexpres1}). In this circumstance, we utilize either $Q_{0}(x)\equiv 1$ or $Q_{1}(x)$, together with $Q_{p}(x)$ to compute $Q_{p-1}(x)$ based on (\ref{downexpre}). Since
\begin{equation}
   \label{q0rel}
   Q_{0}(x)=\psi_{0}^{(p)}(x)Q_{p}(x)+\psi_{0}^{(p-1)}(x)Q_{p-1}(x),
\end{equation}
we can obtain
\begin{equation}
	\label{qn1v0}
	Q_{p-1}(x)=\frac{Q_{0}(x)-\psi_{0}^{(p)}(x)Q_{p}(x)}{\psi_{0}^{(p-1)}(x)}.
\end{equation}

(\ref{downexpres1}) shows that $Q_{p-t}(x)$ is independent of $Q_{p-1}(x)$. When $p-t=0$, $Q_{p-t}(x)$ is substantially $Q_{0}(x)$. According to (\ref{downexpres1}), the relation holds
\begin{equation}
	\label{q0expre}
	Q_{0}(x)=\psi_{0}^{(p)}(x)Q_{p}(x),
\end{equation}
with $\psi_{0}^{(p-1)}(x)=0$. Therefore, (\ref{qn1v0}) fails to compute $Q_{p-1}(x)$ under this condition. We then choose $Q_{1}(x)$ instead of $Q_{0}(x)$ to calculate $Q_{p-1}(x)$. According to Lemma \ref{wtpro}, $\psi_{1}^{(p-1)}(x)\neq 0$ must hold when $\psi_{0}^{(p-1)}(x)=0$. Therefore, $Q_{p-1}(x)$ can be acquired from (\ref{downexpre}) as
\begin{equation}
	\label{qn1v1}
	Q_{p-1}(x)=\frac{Q_{1}(x)-\psi_{1}^{(p)}(x)Q_{p}(x)}{\psi_{1}^{(p-1)}(x)}.
\end{equation}
Likewise, when $\psi_{1}^{(p-1)}(x)=0$, we use (\ref{qn1v0}) to compute $Q_{p-1}(x)$.

Once $Q_{p-1}(x)$ is known, a new three-term recurrence relation can be given under the condition that $Q_{p-t}(x)\neq 0$ as
\begin{equation}
	\label{abi_recs3}
	Q_{p+s}(x)=\Psi_{p+s}^{(p)}(x)Q_{p}(x)+L(x)Q_{p-t}(x),
\end{equation}
where
\begin{equation}
	\label{Lx}
	L(x)=\frac{\Psi_{p+s}^{(p-1)}(x)Q_{p-1}(x)}{Q_{p-t}(x)}.
\end{equation}

When $Q_{p-t}(x)=0$, we can infer from (\ref{downexpres1}) that $\psi_{p-t}^{(p)}(x)\neq 0$, and then $Q_{p}(x)=0$ according to Lemma \ref{wtpro}. The condition $Q_{p}(x)=0$ implies that $Q_{p-1}(x)\neq 0$; otherwise it would lead to $Q_{0}(x)=0$, which contradicts the initial condition $Q_{0}(x)\equiv 1$. Thus, (\ref{upexpre}) is reduced to
\begin{equation}
	\label{upexpres2}
	Q_{p+s}(x)=\Psi_{p+s}^{(p-1)}(x)Q_{p-1}(x)\neq 0.
\end{equation}
Since both $Q_{p}(x)=0$ and $Q_{p-t}(x)=0$, while $Q_{p+s}(x)\neq 0$, it is impossible to compute $Q_{p+s}(x)$ from $Q_{p}(x)$ and $Q_{p-t}(x)$ via a standard three-term recurrence relation. We introduce a bias term $b(x)$ to handle such problem:
\begin{equation}
	\label{upexpres3}
	Q_{p+s}(x)= l_{1}\cdot Q_{p}(x)+l_{2}\cdot Q_{p-t}(x)+b(x),
\end{equation}
where
\begin{equation}
	\label{biasb}
	b(x)=\Psi_{p+s}^{(p-1)}(x)Q_{p-1}(x),
\end{equation}
and $l_{1}$ and $l_{2}$ are arbitrary real numbers.
	
\section{Recurrence relations with different directions}
\label{sec_direction}
Section \ref{sec_arbrecurr} discusses the computation of $Q_{p+s}(x)$ from both $Q_{p}(x)$ and $Q_{p-t}(x)$, exhibiting the classical recursive formula for increasing degree. In this section, we analyze two other recurrence directions. One is the ``degree decrease'' direction, which computes $Q_{p-t}(x)$ from $Q_{p+s}(x)$ and $Q_{p}(x)$; the other is the ``end-to-middle'' one, which calculates $Q_{p}(x)$ from both $Q_{p+s}(x)$ and $Q_{p-t}(x)$.

\subsection{Degree decrease}
According to (\ref{fracrelat}), we can obtain
\begin{equation}
	\label{dereduc}
	Q_{p-t}(x)=\hat{M}(x)Q_{p+s}(x)+\hat{N}(x)Q_{p}(x),
\end{equation}
where the coefficients $\hat{M}(x)$ and $\hat{N}(x)$ are calculated by
\begin{equation}
	\label{coeffdecr}
	\left\{
	\begin{aligned}
		\hat{M}(x)=&\frac{\psi_{p-t}^{(p-1)}(x)}{\Psi_{p+s}^{(p-1)}(x)}\\
		\hat{N}(x)=&-\frac{\mathrm{det}(\mathbf{K})}{\Psi_{p+s}^{(p-1)}(x)}\\
	\end{aligned}
	\right..
\end{equation}

The first special case is when $\psi_{p-t}^{(p-1)}(x)=0$, where (\ref{downexpre}) reduces to (\ref{downexpres1}). It is straightforward to determine the three-term recurrence relation as
\begin{equation}
	\label{decrsp1}
	Q_{p-t}(x)=\psi_{p-t}^{(p)}(x)Q_{p}(x)+0\cdot Q_{p+s}(x).
\end{equation}

The second special case is when $\Psi_{p+s}^{(p-1)}(x)=0$ and $\psi_{p-t}^{(p-1)}(x)\neq 0$. In this case, we compute $Q_{p-t}(x)$ from (\ref{downexpre}).  $Q_{p-1}(x)$ is computed from $Q_{p}(x)$ and either $Q_{0}(x)$ or $Q_{1}(x)$, similar to the strategy described in (\ref{qn1v0}) and (\ref{qn1v1}).  Once $Q_{p-1}(x)$ is known and $Q_{p+s}(x)\neq 0$, the three-term recurrence relation for this case is
\begin{equation}
	\label{decre_sp2}
	Q_{p-t}(x)=\psi_{p-t}^{(p)}(x)Q_{p}(x)+\hat{L}(x)Q_{p+s}(x),
\end{equation}
where
\begin{equation}
	\label{hatLx}
	\hat{L}(x)=\frac{\psi_{p-t}^{(p-1)}(x)Q_{p-1}(x)}{Q_{p+s}(x)}.
\end{equation}

When $Q_{p+s}(x)=0$ and $\Psi_{p+s}^{(p-1)}(x)=0$, we can make sure that $\Psi_{p+s}^{(p)}(x)\neq 0$ according to Lemma \ref{wspro}. Therefore, this yields $Q_{p}(x)=0$, which also guarantees that $Q_{p-1}(x)\neq 0$. Similarly to (\ref{upexpres3}), the relation for this case is
\begin{equation}
	\label{downexpres3}
	Q_{p-t}(x)= \hat{l}_{1}\cdot Q_{p+s}(x)+\hat{l}_{2}\cdot Q_{p}(x)+\hat{b}(x),
\end{equation}
where the bias term is
\begin{equation}
	\label{biasb}
	\hat{b}(x)=\psi_{p-t}^{(p-1)}(x)Q_{p-1}(x)\neq 0,
\end{equation}
and the coefficients $\hat{l}_{1}$ and $\hat{l}_{2}$ are arbitrary real numbers.

\subsection{End-to-middle}
The relation that expresses $Q_{p}(x)$ in terms of $Q_{p+s}(x)$ and $Q_{p-t}(x)$ can also be derived from (\ref{fracrelat})
\begin{equation}
	\label{endmidrec}
	Q_{p}(x)=\tilde{M}(x)Q_{p+s}(x)+\tilde{N}(x)Q_{p-t}(x),
\end{equation}
where the coefficients $\tilde{M}(x)$ and $\tilde{N}(x)$ have the explicit expressions
\begin{equation}
	\label{coeffendmid}
	\left\{
	\begin{aligned}
		\tilde{M}(x)=&\frac{\psi_{p-t}^{(p-1)}(x)}{\mathrm{det}(\mathbf{K})}\\
		\tilde{N}(x)=&-\frac{\Psi_{p+s}^{(p-1)}(x)}{\mathrm{det}(\mathbf{K})}\\
	\end{aligned}
	\right..
\end{equation}
Note that (\ref{coeffendmid}) is valid under the conditions that $\Psi_{p+s}^{(p-1)}(x)\neq 0$, $\psi_{p-t}^{(p-1)}(x)\neq 0$ and $\mathrm{det}(\mathbf{K})\neq 0$.

There are three special cases for this recursive direction. We provide the detailed analysis and the corresponding resolutions for these cases.

(I) The condition that $\Psi_{p+s}^{(p-1)}(x)=0$. Under this condition, (\ref{upexpre}) reduces to (\ref{upexpres1}). From Lemma \ref{wspro}, we know that $\Psi_{p+s}^{(p)}(x)\neq 0$. Thus, the three-term recurrence relation for this case is
\begin{equation}
	\label{endspec1}
	Q_{p}(x)=\frac{1}{\Psi_{p+s}^{(p)}(x)}Q_{p+s}(x)+0\cdot Q_{p-t}(x).
\end{equation}

(II) The condition that $\psi_{p-t}^{(p-1)}(x)=0$ and $\Psi_{p+s}^{(p-1)}(x)\neq 0$. Under this condition, (\ref{downexpre}) is changed into (\ref{downexpres1}). Similarly, Lemma \ref{wtpro} determines that $\psi_{p-t}^{(p)}(x)\neq 0$. The recursive formula is therefore
\begin{equation}
	\label{endspec2}
	Q_{p}(x)=\frac{1}{\psi_{p-t}^{(p)}(x)}Q_{p-t}(x)+0\cdot Q_{p+s}(x).
\end{equation}

(III) The condition that $\mathrm{det}(\mathbf{K})=0$ while both $\Psi_{p+s}^{(p-1)}(x)\neq 0$ and $\psi_{p-t}^{(p-1)}(x)\neq 0$. Under this condition,
there are two opposite cases. 

(III-A) Both $\Psi_{p+s}^{(p)}(x)=0$ and $\psi_{p-t}^{(p)}(x)=0$. In this case, (\ref{upexpre}) and (\ref{downexpre}) become respectively
\begin{equation}
	\label{upexpresp3}
	Q_{p+s}(x)=\Psi_{p+s}^{(p-1)}(x)Q_{p-1}(x),
\end{equation} 
and
\begin{equation}
	\label{downexpresp3}
	Q_{p-t}(x)=\psi_{p-t}^{(p-1)}(x)Q_{p-1}(x).
\end{equation}
Thus, we can get $Q_{p-1}(x)$ from either $Q_{p+s}(x)$ or $Q_{p-t}(x)$. For example, we use $Q_{p+s}(x)$ to compute
\begin{equation}
	\label{Qnm1}
	Q_{p-1}(x)=\frac{Q_{p+s}(x)}{\Psi_{p+s}^{(p-1)}(x)}.
\end{equation}
With the help of (\ref{downexpre}), and either $Q_{0}(x)$ or $Q_{1}(x)$, we are able to obtain $Q_{p}(x)$ through
\begin{equation}
	\label{Q_nspec1}
	Q_{p}(x)=\frac{Q_{0}(x)-\psi_{0}^{(p-1)}(x)Q_{p-1}(x)}{\psi_{0}^{(p)}(x)}.
\end{equation}

When $Q_{p+s}(x)\neq 0$, it follows from (\ref{downexpresp3}) that $Q_{p-t}(x)\neq 0$. According to (\ref{Q_nspec1}) as well as (\ref{downexpresp3}), the three-term recurrence relation is
\begin{equation}
	\label{Qn_rec3a}
	Q_{p}(x)=\tilde{L_{1}}(x)Q_{p+s}(x)+\tilde{L_{2}}(x)Q_{p-t}(x),
\end{equation}
where
\begin{equation}
	\label{tildeL1}
	\tilde{L_{1}}(x)=\frac{Q_{0}(x)}{\psi_{0}^{(p)}(x)Q_{p+s}(x)},
\end{equation}
and
\begin{equation}
	\label{tildeL2}
	\tilde{L_{2}}(x)=-\frac{\psi_{0}^{(p-1)}(x)}{\psi_{0}^{(p)}(x)\psi_{p-t}^{(p-1)}(x)}.
\end{equation}

When $Q_{p+s}(x)=0$, it follows that $Q_{p-1}(x)=0$, which leads to $Q_{p-t}(x)=0$ as well. The recursive formula is then
\begin{equation}
	\label{Qn_rec3b}
	Q_{p}(x)=\tilde{l}_{11}Q_{p+s}(x)+\tilde{l}_{12}Q_{p-t}(x)+\tilde{b_{1}}(x),
\end{equation}
where the bias term is
\begin{equation}
	\label{bias3b}
	\tilde{b_{1}}(x)=\frac{Q_{0}(x)-\psi_{0}^{(p-1)}(x)Q_{p-1}(x)}{\psi_{0}^{(p)}(x)},
\end{equation}
and $\tilde{l}_{11}$ and $\tilde{l}_{12}$ are arbitrary real numbers.

(III-B) Both $\Psi_{p+s}^{(p)}(x)\neq 0$ and $\psi_{p-t}^{(p)}\neq 0$. In this case, although both (\ref{upexpre}) and (\ref{downexpre}) are established, yet they are linearly dependent in fact. We use $Q_{p-t}(x)$ and either $Q_{0}(x)$ or $Q_{1}(x)$ to compute $Q_{p-t+1}(x)$ via (\ref{downexpre})
\begin{equation}
	\label{Qnmtm1}
	Q_{p-t+1}(x)=\frac{Q_{0}(x)-\psi_{0}^{(p-t)}(x)Q_{p-t}(x)}{\psi_{0}^{(p-t+1)}(x)}.
\end{equation}
Subsequently, we compute $Q_{p}(x)$ from $Q_{p-t}(x)$ and $Q_{p-t+1}(x)$ based on (\ref{upexpre}):
\begin{equation}
	\label{Qn}
	Q_{p}(x)=\Psi_{p}^{(p-t+1)}(x)Q_{p-t+1}(x)+\Psi_{p}^{(p-t)}(x)Q_{p-t}(x).
\end{equation}

When $Q_{p+s}(x)\neq 0$, the three-term recursive formula is
\begin{equation}
	\label{Qn_rec3c}
	Q_{p}(x)=\Psi_{p}^{(p-t)}(x)Q_{p-t}(x)+\tilde{L_{3}}(x)Q_{p+s}(x),
\end{equation}
where
\begin{equation}
	\label{tildeL3}
	\tilde{L_{3}}(x)=\frac{\Psi_{p}^{(p-t+1)}(x)Q_{p-t+1}(x)}{Q_{p+s}(x)}.
\end{equation}

When $Q_{p+s}(x)=0$, since (\ref{upexpre}) and (\ref{downexpre}) are dependent, we can derive that $Q_{p-t}(x)=0$, which reduces (\ref{Qn}) to
\begin{equation}
	\label{Qnspec}
	Q_{p}(x)=\Psi_{p}^{(p-t+1)}(x)Q_{p-t+1}(x).
\end{equation}
The recurrence relation is therefore
\begin{equation}
	\label{Qn_rec3d}
	Q_{p}(x)=\tilde{l}_{21}Q_{p+s}(x)+\tilde{l}_{22}Q_{p-t}(x)+\tilde{b_{2}}(x),
\end{equation}
where the bias for this case is
\begin{equation}
	\label{bias3c}
	\tilde{b_{2}}(x)=\Psi_{p}^{(p-t+1)}(x)Q_{p-t+1}(x),
\end{equation}
and $\tilde{l}_{21}$ and $\tilde{l}_{22}$ are arbitrary real numbers.

\subsection{Remarks}
We give some remarks for the above analysis and the obtained recurrence relations.

Firstly, some classes of orthogonal polynomials possess the intrinsic parameters, such as Gegenbauer and Jacobi polynomials. We represent the intrinsic parameters by $\mathbf{\Theta}$. Consequently, both the polynomials and the coefficients are substantially with respect to both $x$ and $\mathbf{\Theta}$. However, we omit $\mathbf{\Theta}$ in the mathematical presentations for conciseness and efficiency.

Secondly, the standard three-term recurrence relation formulated in (\ref{Favard}) is  a special case of (\ref{abi_rec}), corresponding to  $(s=1,t=1)$. Furthermore, $s=t$ supports consecutive recursive computation; $s\neq t$, however, fails to proceed continuous computation.  

Thirdly, the parameter setting $s\neq t$ probably produces the ``bias terms'', resulting in non-rigid three-term recurrence relations. Nevertheless, they are reasonable especially when $(p, s, t)$ are freely selected. Taking Legendre polynomial as an example, the polynomials of odd degrees have zero values at $x_{0}=0$, while those of even degrees have non-zero values at that point. When $(p=5,s=3,t=2)$ and the computation of degree increase is required, it is impossible to find a rigid three-term relation to compute $L_{8}(x_{0})\neq 0$ from two zero-value polynomials $(L_{5}(x_{0})=0,L_{3}(x_{0})=0)$. Such phenomena occur merely because the random selection of $(p,s,t)$ breaks the continuity of recursive formulas.

Finally, according to Theorem \ref{mainthm}, we can derive the explicit recurrence relations for different families of orthogonal polynomials with algebraic techniques, theoretically. When $(s,t)$ are small, e.g., equal to $2$, $3$ even $4$, such a strategy is feasible (refer to Section \ref{sec_inst2} for $(s=2, t=2)$), although it resorts to brute force which is arduous indeed. However, it should be noted that most special cases that we discussed above, such as divide-by-zero and zero-coefficient, could be partly avoided in this way. When $(s,t)$ are large, the computation must resort to programming, which makes it necessary to discuss the special cases corresponding to different recursive computation directions.

\section{Instantiation for 2-degree steps}
\label{sec_inst2}
% In the condition that efficient computation and economic storage are critical considerations, the recurrence relations with multi-degree step become preferable, because it is unnecessary to compute the complete set of polynomials, which saves not only the computation time but also the storage space. 
Hermite, Gegenbauer and Legendre polynomials are typical representatives of the classical orthogonal polynomials. Moreover, they have broad applications in mathematics, physics and engineering fields. In this section, we derive the explicit recurrence relations with 2-degree steps, i.e., $(s=2, t=2)$, for these three classes of orthogonal polynomials. 
\subsection{Hermite polynomials}
According to Theorem \ref{thm1} and (\ref{herrecur}), Favard's coefficients of Hermite polynomials are 
\begin{equation}
	\label{hercoef}
	\left\{
	\begin{aligned}
	A_{p}&=2\\
	B_{p}&=0\\
	C_{p}&=2p\\
	\end{aligned}
	\right.  \quad p\geq 1.
\end{equation}
Substituting $p+1$ for $p$ in (\ref{herrecur}), we can get
\begin{equation}
	\label{h2}
	H_{p+2}(x)=2xH_{p+1}(x)-2(p+1)H_{p}(x).
\end{equation}
Still substituting (\ref{herrecur}) into (\ref{h2}) yields
\begin{equation}
	\label{h2fin}
		H_{p+2}(x)=\big(4x^{2}-2p-2\big)H_{p}(x)-4pxH_{p-1}(x).
\end{equation}
Therefore, the coefficients for $s=2$ are
\begin{equation}
	\label{hs2coef}
	s=2:
	\left\{
	\begin{aligned}
	\Psi_{p+2}^{(p)}(x)&=4x^{2}-2p-2\\
	 \Psi_{p+2}^{(p-1)}(x)&=-4px\\
	\end{aligned}
	\right..
\end{equation}
Similarly,  it is easy to get the expression for the case that $t=2$.
\begin{equation}
	\label{h2t}
	H_{p-2}(x)=\frac{x}{p-1}H_{p-1}(x)-\frac{1}{2(p-1)}H_{p}(x).
\end{equation}
Obviously,
\begin{equation}
	\label{ht2coef}
	t=2:
	\left\{
	\begin{aligned}
	\psi_{p-2}^{(p)}(x)&=-\frac{1}{2(p-1)}\\
	 \psi_{p-2}^{(p-1)}(x)&=\frac{x}{p-1}\\   
	\end{aligned}
	\right..
\end{equation}
 Substituting (\ref{hs2coef}) and (\ref{ht2coef}) into (\ref{stepcoeff}), we can obtain the recurrence relation with 2-degree steps for Hermite polynomials based on (\ref{abi_rec})
\begin{equation}
	\label{her2recur}
	 H_{p+2}(x)=2(2x^{2}-2p-1)H_{p}(x)-4p(p-1)H_{p-2}(x),\quad p\geq 2
\end{equation}
with $H_{0}(x)=1$, $H_{1}(x)=2x$, $H_{2}(x)=4x^{2}-2$ and $H_{3}(x)=8x^{3}-12x$.

It should be noted that, in our recent publication in the domain of computer science\cite{yang2026}, we obtained the same result as (\ref{her2recur}) in another way. However, it is that discovery regarding Hermite polynomials that gives us the inspiration to investigate the generalized recurrence relations discussed in this paper.

\subsection{Gegenbauer polynomials}
According to Theorem \ref{thm1} and (\ref{gegrecur}), Favard's coefficients of Gegenbauer polynomials are 
\begin{equation}
	\label{gencoef}
	\left\{
	\begin{aligned}
	A_{p}&=\frac{2(p+\lambda)}{p+1}\\
	B_{p}&=0\\
	C_{p}&=\frac{p+2\lambda-1}{p+1}\\
	\end{aligned}
	\right.  \quad p\geq 1.
\end{equation}

First, we discuss the case that $s=2$. According to (\ref{gegrecur}), we have
\begin{equation}
\label{gs2recu}
G_{p+2}^{(\lambda)}(x)=\frac{2(p+\lambda+1)}{p+2}xG_{p+1}^{(\lambda)}(x)-\frac{p+2\lambda}{p+2}G_{p}^{(\lambda)}(x).
\end{equation}
Substituting (\ref{gegrecur}) into (\ref{gs2recu}), we obtain the explicit expression below:
\begin{equation}
	\label{gs2}
	\begin{aligned}
	G_{p+2}^{(\lambda)}(x)=&\frac{4(p+\lambda)(p+\lambda+1)x^{2}-(p+1)(p+2\lambda)}{(p+1)(p+2)}G_{p}^{(\lambda)}(x)\\
	                                                   -&\frac{2(p+\lambda+1)(p+2\lambda-1)}{(p+1)(p+2)}xG_{p-1}^{(\lambda)}(x)
	\end{aligned}.
\end{equation}
Apparently, the coefficients for $s=2$ are
\begin{equation}
	\label{gs2coef}
	s=2:
	\left\{
	\begin{aligned}
	\Psi_{p+2}^{(p)}(x)=&\frac{4(p+\lambda)(p+\lambda+1)x^{2}-(p+1)(p+2\lambda)}{(p+1)(p+2)}\\ \Psi_{p+2}^{(p-1)}(x)=&-\frac{2(p+\lambda+1)(p+2\lambda-1)}{(p+1)(p+2)}x\\
	\end{aligned}
	\right..
\end{equation}

Similarly,  it is also necessary to get the expression for the case that $t=2$. From (\ref{gegrecur}), we can rewrite the expression to
\begin{equation}
	\label{gt1}
	G_{p-1}^{(\lambda)}(x)=\frac{2(p+\lambda)}{p+2\lambda-1}xG_{p}^{(\lambda)}(x)-\frac{p+1}{p+2\lambda-1}G_{p+1}^{(\lambda)}(x).
\end{equation}
Then, substituting $p-1$ for $p$ in (\ref{gt1}) gives the final expression as 
\begin{equation}
\label{gt2}
G_{p-2}^{(\lambda)}(x)=\frac{2(p+\lambda-1)}{p+2\lambda-2}xG_{p-1}^{(\lambda)}(x)-\frac{p}{p+2\lambda-2}G_{p}^{(\lambda)}(x).
\end{equation}
The coefficients are
\begin{equation}
	\label{gt2coef}
	t=2:
	\left\{
	\begin{aligned}
	\psi_{p-2}^{(p)}(x)&=-\frac{p}{p+2\lambda-2}\\
	\psi_{p-2}^{(p-1)}(x)&=\frac{2(p+\lambda-1)}{p+2\lambda-2}x\\   
	\end{aligned}
	\right..
\end{equation}
 
 Substituting (\ref{gs2coef}) and (\ref{gt2coef}) into (\ref{stepcoeff}), we can finally obtain the explicit recurrence relation with 2-degree steps for Gegenbauer polynomials:
\begin{equation}
	\label{geg2recur}
	\begin{aligned}
		G_{p+2}^{(\lambda)}(x)&=K(p,\lambda)\Big(4(p+\lambda-1)(p+\lambda)(p+\lambda+1)x^{2}\\
		                                                 &-p(p+\lambda+1)(p+2\lambda-1)\\
		                                                 &-(p+1)(p+\lambda-1)(p+2\lambda)\Big)G_{p}^{(\lambda)}(x)\\
	                                                     &-K(p,\lambda)(p+\lambda+1)(p+2\lambda-1)(p+2\lambda-2)G_{p-2}^{(\lambda)}(x),\\
	\end{aligned}\quad p\geq 2
\end{equation}
where
\begin{equation}
	\label{genden}
	K(p,\lambda)=\frac{1}{(p+1)(p+2)(p+\lambda-1)}.
\end{equation}
The initial conditions for the recurrence relation are $G_{0}^{(\lambda)}(x)=1$, $G_{1}^{\lambda}(x)=2\lambda x$,  $G_{2}^{(\lambda)}(x)=2\lambda(\lambda+1)x^{2}-\lambda$, and $G_{3}^{(\lambda)}(x)=\frac{4}{3}\lambda(\lambda+1)(\lambda+2)x^{3}-2\lambda(\lambda+1)x$.

\subsection{Legendre polynomials}
According to Theorem \ref{thm1} and (\ref{legrecur}), Favard's coefficients of Legendre polynomials are
\begin{equation}
	\label{legcoef}
	\left\{
	\begin{aligned}
	A_{p}&=\frac{2p+1}{p+1}\\
	B_{p}&=0\\
	C_{p}&=\frac{p}{p+1}\\
	\end{aligned}
	\right.,  \quad p\geq 1.
\end{equation}
Substituting $p+1$ for $p$ in (\ref{legrecur}) yields
\begin{equation}
	\label{l2}
	L_{p+2}(x)=\frac{2p+3}{p+2}xL_{p+1}(x)-\frac{p+1}{p+2}L_{p}(x).
\end{equation}
Then, we simplify (\ref{l2}) with the help of (\ref{legrecur}) and obtain the final expression
\begin{equation}
	\label{l2fin}
	\begin{aligned}
		L_{p+2}(x)&=\left(\frac{(2p+1)(2p+3)}{(p+1)(p+2)}x^{2}-\frac{p+1}{p+2}\right)L_{p}(x)\\
		                      &-\frac{p(2p+3)}{(p+1)(p+2)}xL_{p-1}(x)\\
	\end{aligned}.
\end{equation}
The coefficients corresponding to $s=2$ can be learned from (\ref{l2fin}):
\begin{equation}
	\label{ls2coef}
	s=2:
	\left\{
	\begin{aligned}
	\Psi_{p+2}^{(p)}(x)&=\frac{(2p+1)(2p+3)}{(p+1)(p+2)}x^{2}-\frac{p+1}{p+2}\\
	\Psi_{p+2}^{(p-1)}(x)&=-\frac{p(2p+3)}{(p+1)(p+2)}x\\
	\end{aligned}
	\right..
\end{equation}

Subsequently, we derive the expression for the case that $t=2$. According to (\ref{legrecur}), it is easy to obtain
\begin{equation}
	\label{l2t}
	L_{p-2}(x)=\frac{2p-1}{p-1}xL_{p-1}(x)-\frac{p}{p-1}L_{p}(x).
\end{equation}
Obviously, the coefficients are
\begin{equation}
	\label{lt2coef}
	t=2:
	\left\{
	\begin{aligned}
	\psi_{p-2}^{(p)}(x)&=-\frac{p}{p-1}\\
	\psi_{p-2}^{(p-1)}(x)&=\frac{2p-1}{p-1}x\\   
	\end{aligned}
	\right..
\end{equation}

Substituting (\ref{ls2coef}) and (\ref{lt2coef}) into (\ref{stepcoeff}), we can finally obtain the explicit recurrence relation with 2-degree steps for Legendre polynomials:  
\begin{equation}
	\label{leg2recur}
	\begin{aligned}
		L_{p+2}(x)&=\frac{(2p-1)(2p+1)(2p+3)x^{2}-(2p-1)(p+1)^{2}-p^{2}(2p+3)}{(2p-1)(p+1)(p+2)}L_{p}(x)\\
		                      &-\frac{p(p-1)(2p+3)}{(2p-1)(p+1)(p+2)}L_{p-2}(x), \quad\quad p\geq 2
	\end{aligned} 
\end{equation}
with the initial conditions $L_{0}(x)=1$, $L_{1}(x)=x$, $L_{2}(x)=\frac{1}{2}(3x^{2}-1)$ and $L_{3}(x)=\frac{1}{2}(5x^{3}-3x)$.

\section{Precision comparison}
\label{sec_exp}
This section gives some experiments to evaluate the computational precision of the proposed recurrence relations against the standard ones. We define the error $\Delta$ to evaluate the computational precision as
\begin{equation}
	\label{relerr}
	\Delta=\left\{
	\begin{aligned}
	&\left|\frac{Q_{p}^{(s)}(x)-Q_{p}^{(t)}(x)}{Q_{p}^{(t)}(x)}\right|\times 100\%,\; &\, \mathrm{if}\;Q_{p}^{(t)}(x)\neq 0\\
	&\left|Q_{p}^{(s)}(x)-Q_{p}^{(t)}(x)\right|\times 100\%,  &\, \mathrm{if}\;Q_{p}^{(t)}(x)=0\\
	\end{aligned}
	\right.,
\end{equation}
where $Q_{p}^{(t)}(x)$ and $Q_{p}^{(s)}(x)$ denote the polynomials of degree $p$ computed by the standard recurrence relations and the proposed ones, respectively. Hermite, Gegenbauer and Legendre polynomials were also chosen for evaluation. All the code is implemented in MATLAB (version R2024a). 

\subsection{The case $s=t$}
\label{precomp1}
When $s=t$, the proposed recurrence relation (\ref{abi_rec}) is capable of computing the polynomials consecutively. We compute three classes of orthogonal polynomials from $0$ up to the maximum degree of $60$ with different $(s, t)$.

For Hermite polynomials, we set $(s=3, t=3)$ for the proposed recurrence relation (\ref{abi_rec}), while the standard one refers to (\ref{herrecur}). $x$ varies from $-1.5$ to $1.5$ with a step size of $0.01$. Fig. \ref{hercomp} shows the polynomial of degree $21$ (relative low) and the one of degree $57$ (relative high) computed from both recurrence relations. Different line widths are intentionally utilized to avoid entire overlapping of the curves. The errors $\Delta$ are both below $2.0\times10^{-11}\%$ for the polynomials of low and high degrees. 
\begin{figure}[htbp]
\centering
\includegraphics[width=14.0cm]{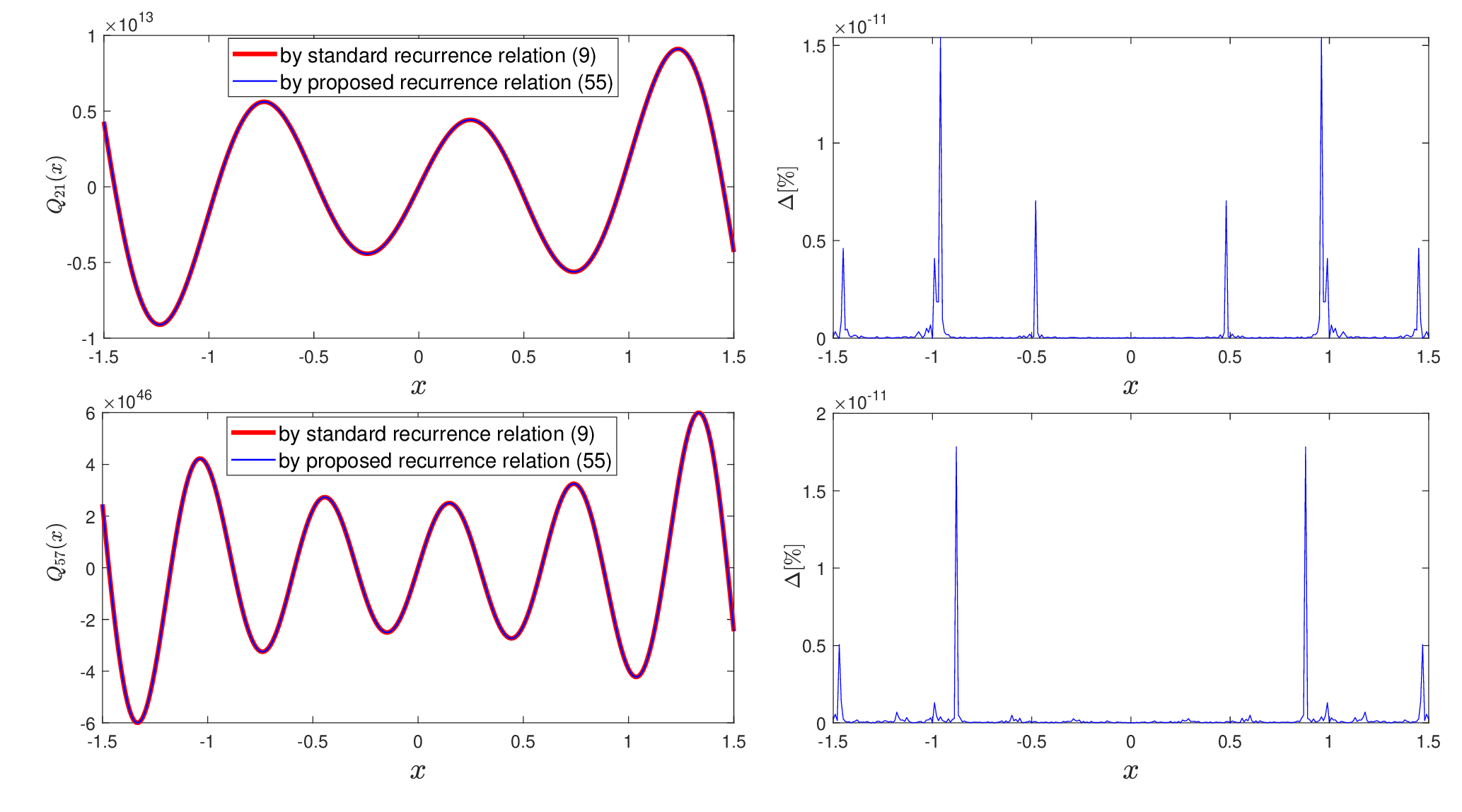}
\caption{Precision comparison with Hermite polynomials. Top left: the polynomials of degree $21$ computed from (\ref{abi_rec}) and (\ref{herrecur}); Top right: the error $\Delta$ for the polynomial of degree $21$; Bottom left: the ones of degree $57$ computed from both recurrence relations; Bottom right: $\Delta$ for the polynomial of degree  $57$.}
\label{hercomp}
\end{figure}

In the computation of Gegenbauer polynomials, $(s=4, t=4)$ are set for the proposed recurrence relation (\ref{abi_rec}). The standard one is formulated in (\ref{gegrecur}). $\lambda$ is set to $3.0$ and $x$ is limited between $-0.9$ and $0.9$ with an increment of $0.01$ to avoid large values. Fig. \ref{gegencomp} shows the polynomial of degree $16$ and the one of degree $56$ computed from both recurrence relations. Obviously, the errors $\Delta$ are both below $4.0\times10^{-10}\%$ for the polynomials of low and high degrees. 
\begin{figure}[htbp]
	\centering
	\includegraphics[width=14.0cm]{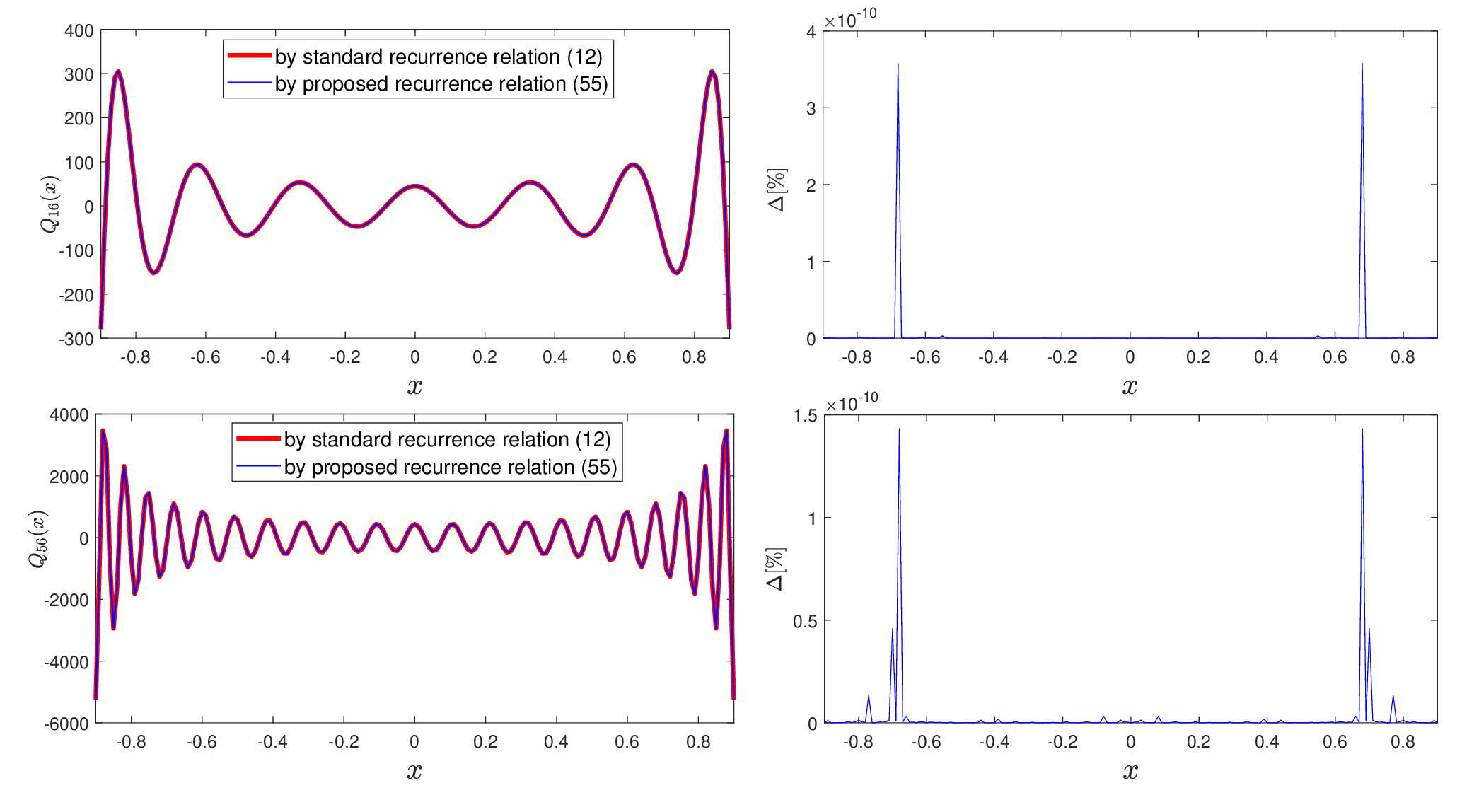}
	\caption{Precision comparison with Gegenbauer polynomials ($\lambda=3.0$). Top left: the polynomials of degree $16$ computed from (\ref{abi_rec}) and (\ref{gegrecur}); Top right: the error $\Delta$ for the polynomial of degree $16$; Bottom left: the ones of degree $56$ computed from both recurrence relations; Bottom right: $\Delta$ for the polynomial of degree $56$.}
	\label{gegencomp}
\end{figure}

In case of Legendre polynomials, we set $(s=5, t=5)$ for the proposed recurrence relation (\ref{abi_rec}). The standard one is denoted as (\ref{legrecur}). $x$ ranges from $-1.0$ to $1.0$ with a stride of $0.01$.  Fig. \ref{legencomp} shows the polynomial of degree $15$ and that of degree $60$ calculated from both recurrence relations. The errors $\Delta$ are below $1.0\times10^{-10}\%$, apparently. 
\begin{figure}[htbp]
	\centering
	\includegraphics[width=14.0cm]{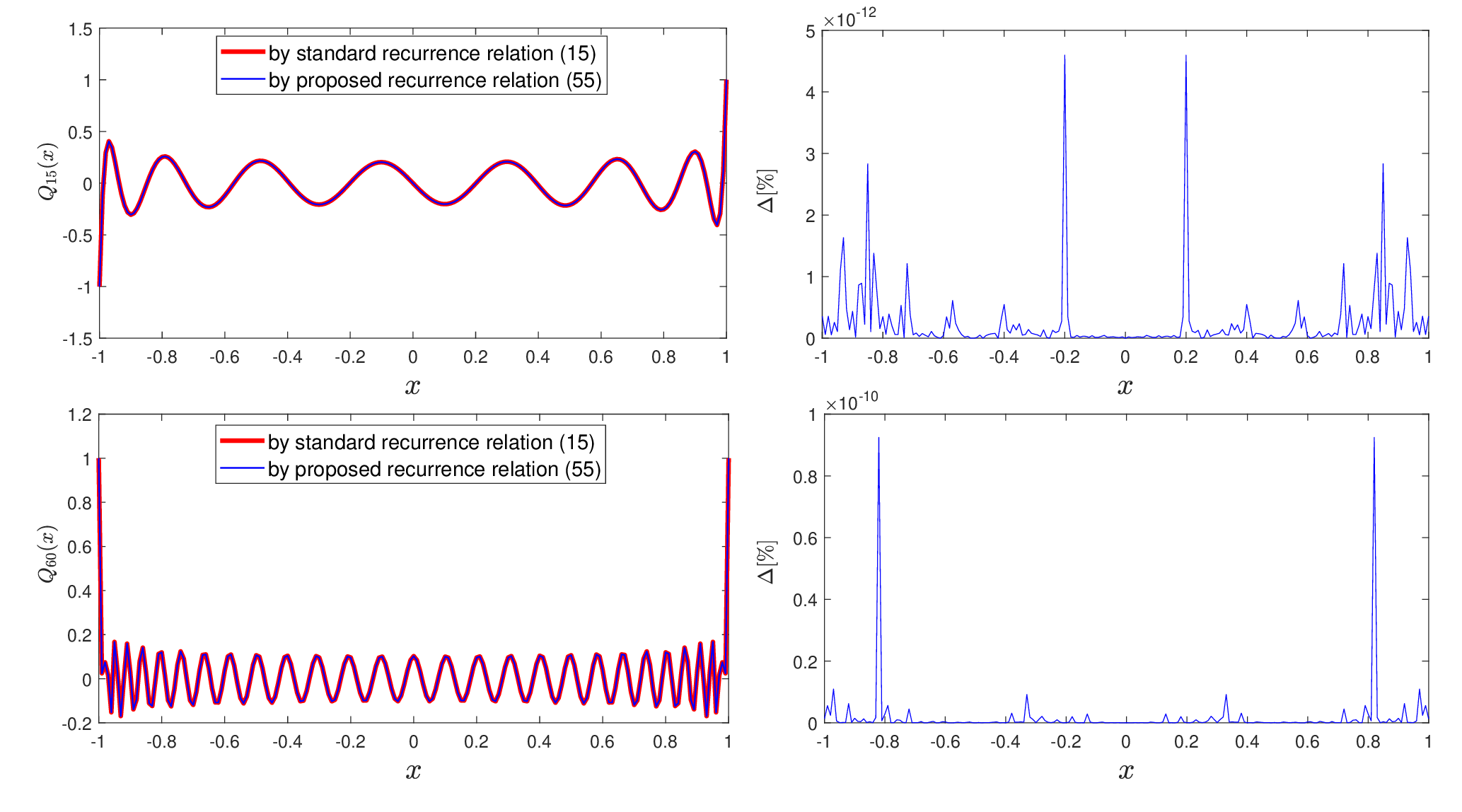}
	\caption{Precision comparison with Legendre polynomials. Top left: the polynomials of degree $15$ computed from (\ref{abi_rec}) and (\ref{legrecur}); Top right: the error $\Delta$ for the polynomial of degree $15$; Bottom left: the ones of degree $60$ computed from both recurrence relations; Bottom right: $\Delta$ for the polynomial of degree $60$.}
	\label{legencomp}
\end{figure}

\subsection{The case $s \neq t$}
\label{precomp2}
In this subsection, we test the proposed recurrence relations for different computational directions. The condition $s\neq t$ confines that the computation can be performed only once. The orthogonal polynomials, the variable $x$ and the intrinsic parameter $\lambda$ are identical to those in Subsection \ref{precomp1}. 

Parameters $(p=17,s=25,t=14)$ and $(p=31,s=32,t=11)$ are set to test (\ref{abi_rec}) for Hermite polynomials. They substantially imply that computing $H_{42}(x)$ from $H_{17}(x)$ and $H_{3}(x)$, and calculating $H_{63}(x)$ based on $H_{31}(x)$ and $H_{20}(x)$, respectively. Fig. \ref{hercomp2} shows the polynomial of degree $42$ and the one of degree $63$ computed from (\ref{abi_rec}) and (\ref{herrecur}). 
\begin{figure}[htbp]
	\centering
	\includegraphics[width=14.0cm]{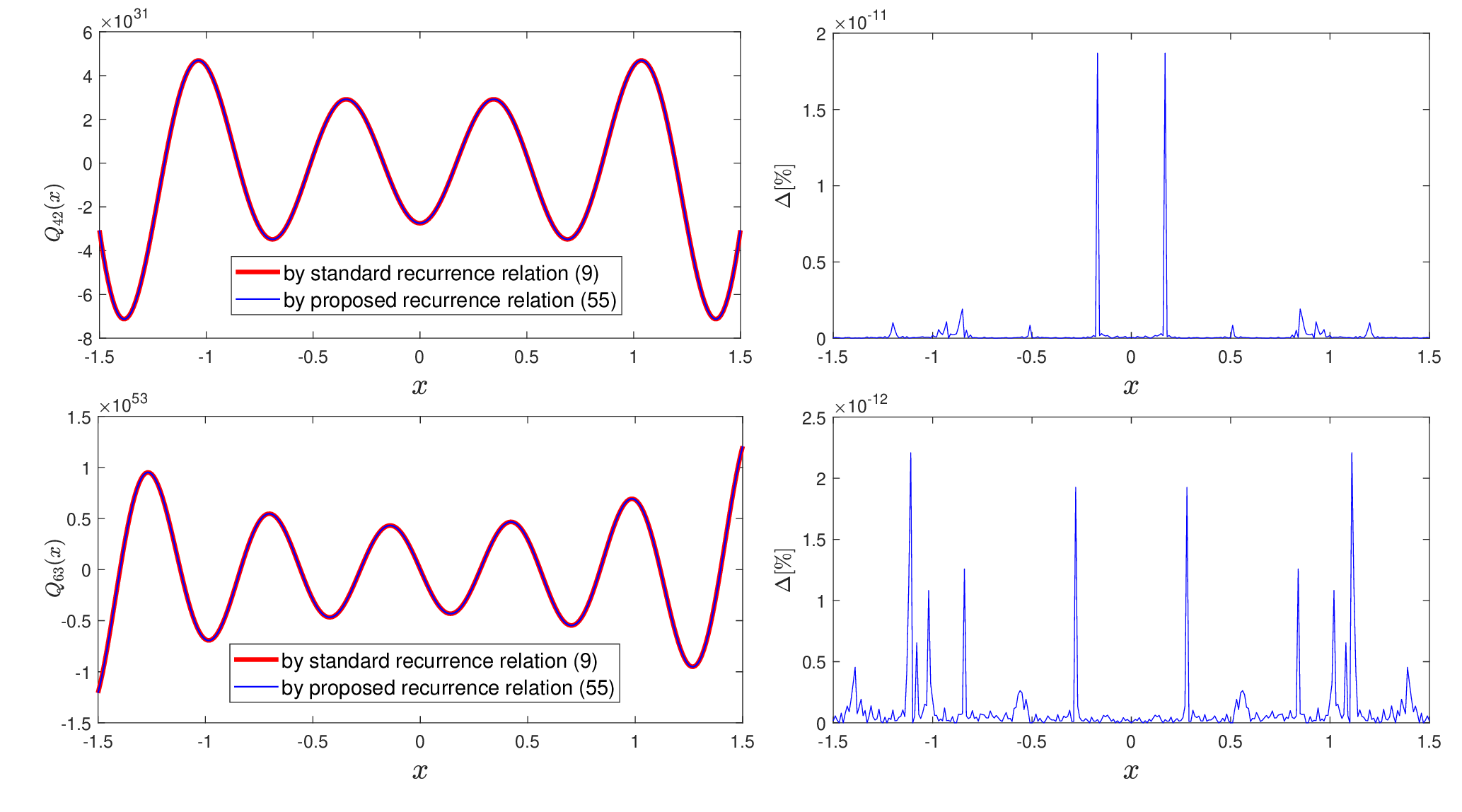}
	\caption{Precision comparison with Hermite polynomials. Top left: the polynomials of degree $42$ computed from (\ref{abi_rec}) and (\ref{herrecur}); Top right: the error $\Delta$ for the polynomial of degree $42$; Bottom left: the ones of degree $63$ computed from both recurrence relations; Bottom right: $\Delta$ for the polynomial of degree $63$.}
	\label{hercomp2}
\end{figure}

Gegenbauer polynomials  ($\lambda=3.0$) are used to test (\ref{dereduc}). The parameters are set to $(p=26,s=15,t=26)$ and $(p=40,s=18,t=29)$. They correspond to computing $G_{0}^{(\lambda)}(x)$ from $G_{41}^{(\lambda)}(x)$ and $G_{26}^{(\lambda)}(x)$, and deriving $G_{11}^{(\lambda)}(x)$ from $G_{58}^{(\lambda)}(x)$ and $G_{40}^{(\lambda)}(x)$. Fig. \ref{gegencomp2} shows the polynomial of degree $0$ and the one of degree $11$ computed via (\ref{dereduc}) and (\ref{gegrecur}).
\begin{figure}[htbp]
	\centering
	\includegraphics[width=14.0cm]{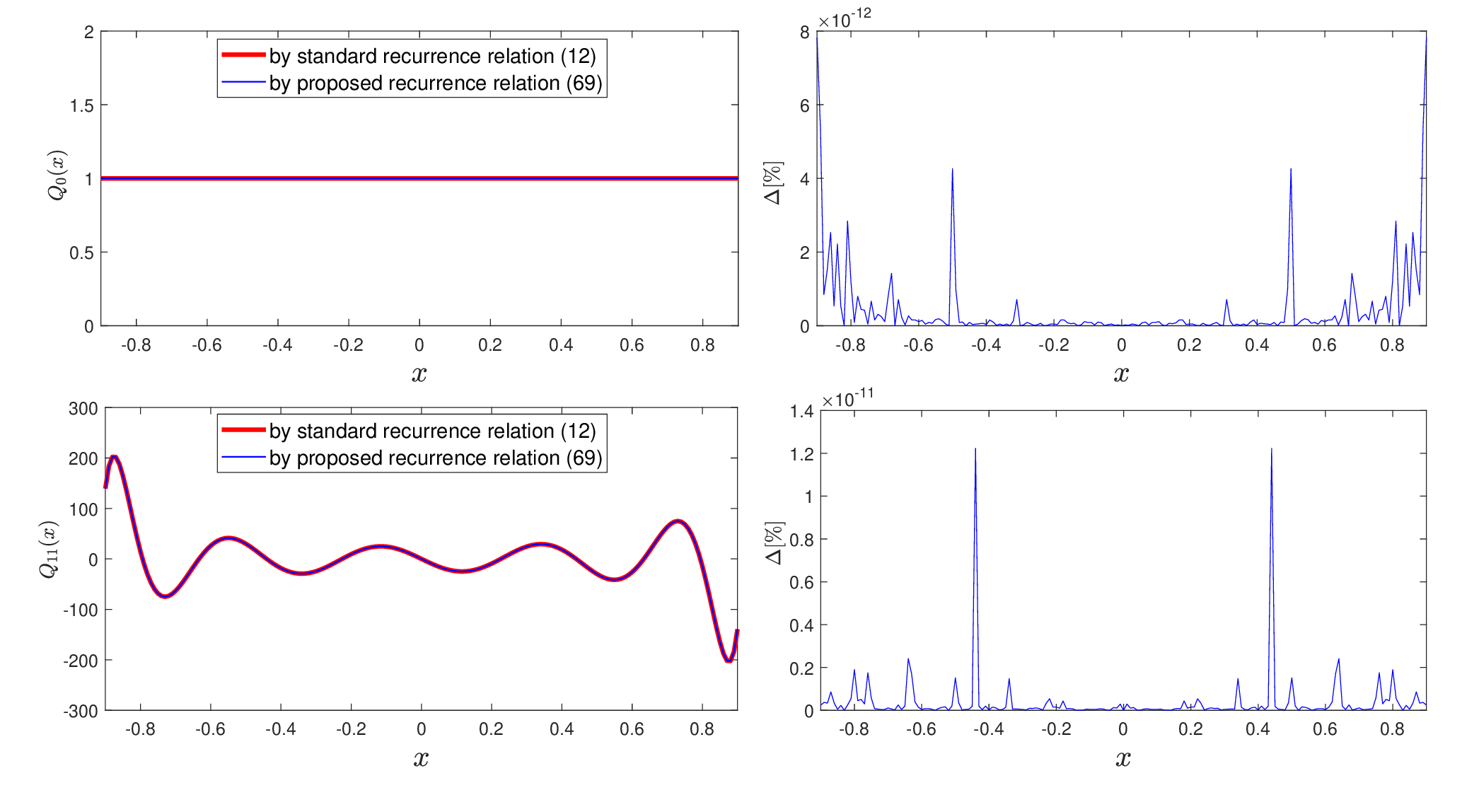}
	\caption{Precision comparison with Gegenbauer polynomials. Top left: the polynomials of degree $0$ computed from the (\ref{dereduc}) and (\ref{gegrecur}); Top right: the error $\Delta$ for the polynomial of degree $0$; Bottom left: the ones of degree $11$ computed from both recurrence relations; Bottom right: $\Delta$ for the polynomial of degree $11$.}
	\label{gegencomp2}
\end{figure}

Legendre polynomials are employed to test (\ref{endmidrec}). The parameters are set to $(p=20,s=31,t=11)$ and $(p=31,s=31,t=18)$. These two cases correspond to computing $L_{20}(x)$ from $L_{51}(x)$ and $L_{9}(x)$, and calculating $L_{31}(x)$ from $L_{62}(x)$ and $L_{13}(x)$. Fig. \ref{legencomp2} shows the polynomial of degree $20$ and the one of degree $31$ from (\ref{endmidrec}) and (\ref{legrecur}).
\begin{figure}[htbp]
	\centering
	\includegraphics[width=14.0cm]{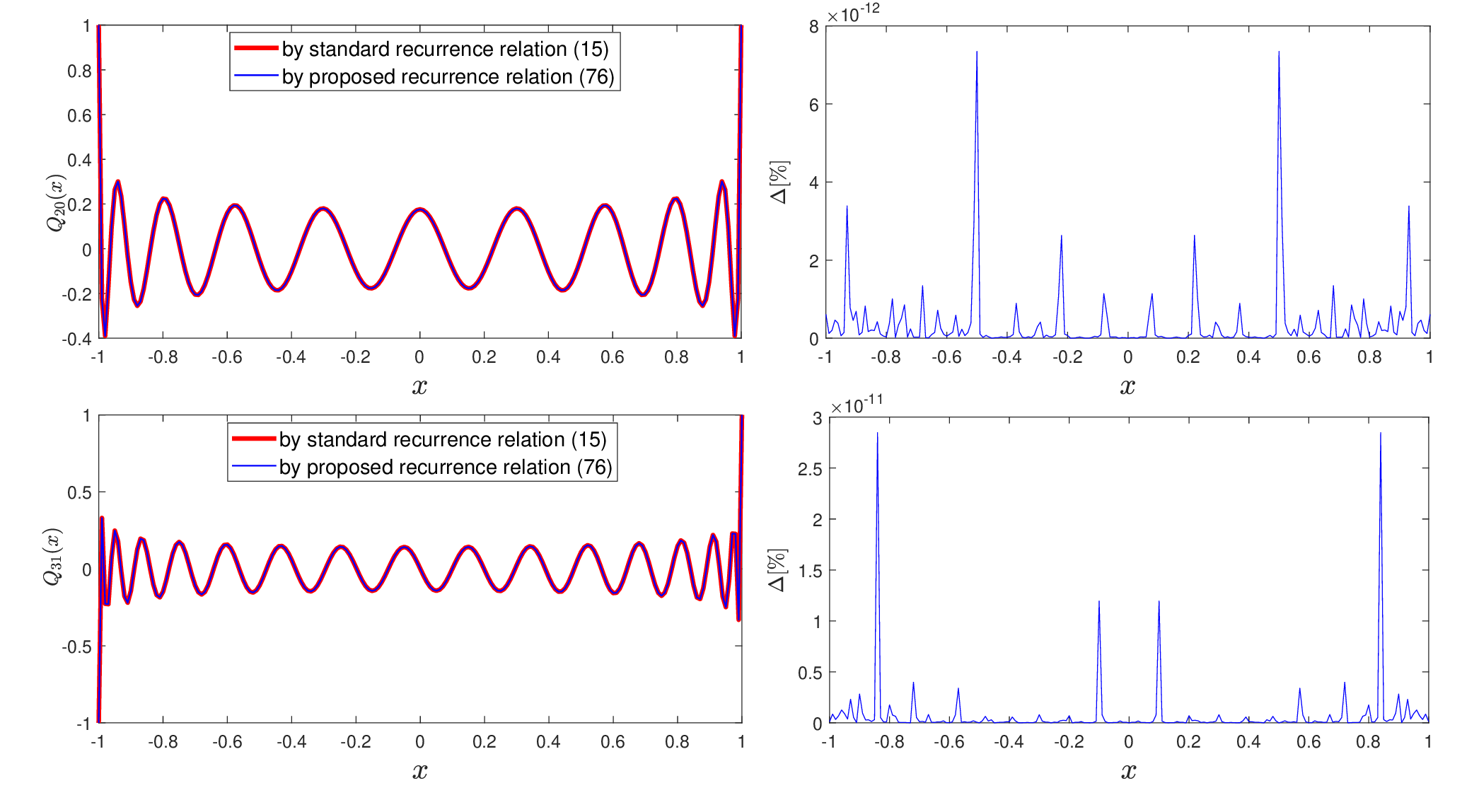}
	\caption{Precision comparison with Legendre polynomials. Top left: the polynomials of degree $20$ computed from (\ref{endmidrec}) and (\ref{legrecur}); Top right: the error $\Delta$ for the polynomial of degree $20$; Bottom left: the ones of degree $31$ computed from both recurrence relations; Bottom right: $\Delta$ for the polynomial of degree $31$.}
	\label{legencomp2}
\end{figure}

All the experimental results show the following facts. First, the proposed recurrence relations (\ref{abi_rec}), (\ref{dereduc}) and (\ref{endmidrec}) are effective for the orthogonal polynomials defined by Favard's theorem. Second, the proposed recurrence relations controlled by $(p, s, t)$ are characterized with degree-skip and direction-adjust capabilities. Last, the proposed recurrence relations almost achieve the same computation precision as the standard recursive relations, which compute the polynomial of any degree strictly based on two adjacent members via rigid 1-degree steps. 

\section{Acknowledgment}
\label{sec:acks}
During the preparation of this article, Bo Yang was partially supported by National Natural Science Foundation of China (grant no. 61502389) and Natural Science Foundation of Shannxi Province (grant no. 2020JM-140).

\bibliography{Recurrence}
\bibliographystyle{elsarticle-num}

\end{document}